\documentclass{amsart}

\usepackage{amsmath}
\usepackage{amssymb}
\usepackage{amsfonts}
\usepackage{graphicx}
\usepackage{texdraw}
\usepackage{graphpap}
\usepackage{enumitem}
\usepackage{color}
\usepackage[OT2,T1]{fontenc}
\usepackage{verbatim} 
\usepackage{multirow}
\usepackage{mathtools}
\usepackage[colorlinks=true, citecolor=blue, filecolor=black, linkcolor=black, urlcolor=black]{hyperref}
\usepackage{bm}
\usepackage{bbm}

\usepackage{mathtools}
\mathtoolsset{showonlyrefs}
\usepackage{todonotes}
\usepackage{appendix}

\theoremstyle{plain}
\newtheorem{thm}{Theorem}[section]

\newtheorem{lem}[thm]{Lemma}

\theoremstyle{remark}

\newtheorem*{rem}{Remark}

\numberwithin{equation}{section}

\newcommand{\R}{{\mathbb R}}

\newcommand{\ZZ}{\mathcal{Z}}

\renewcommand{\d}{{\partial}}

\newcommand{\Lap}{\Delta}

\def\C{\mathbb{C}}

\def\R{\mathbb{R}}

\def\wt{\widetilde}

\numberwithin{equation}{section}

\newcommand{\RN}[1]{%
	\textup{\uppercase\expandafter{\romannumeral#1}}%
}

\begin{document}

\title[Partition functions of determinantal and Pfaffian Coulomb gases]{Partition functions of determinantal and Pfaffian Coulomb gases with radially symmetric potentials}


\author{Sung-Soo Byun}
\address{Center for Mathematical Challenges, Korea Institute for Advanced Study, 85 Hoegiro, Dongdaemun-gu, Seoul 02455, Republic of Korea}
\email{sungsoobyun@kias.re.kr}

\author{Nam-Gyu Kang}
\address{School of Mathematics, Korea Institute for Advanced Study, 85 Hoegiro, Dongdaemun-gu, Seoul 02455, Republic of Korea}
\email{namgyu@kias.re.kr}

\author{Seong-Mi Seo}
\address{Department of Mathematics, Chungnam National University, 99 Daehak-ro, Yuseong-gu, Daejeon 34134, Republic of Korea.}
\email{smseo@cnu.ac.kr}


\date{\today}
\thanks{Sung-Soo Byun and Nam-Gyu Kang were partially supported by Samsung Science and Technology Foundation (SSTF-BA1401-51) and by the National Research Foundation of Korea (NRF-2019R1A5A1028324). 
Sung-Soo Byun was partially supported by a KIAS Individual Grant (SP083201) via the Center for Mathematical Challenges at Korea Institute for Advanced Study.
Nam-Gyu Kang was partially supported by a KIAS Individual Grant (MG058103) at Korea Institute for Advanced Study.
Seong-Mi Seo was partially supported by the National Research Foundation of Korea (2019R1A5A1028324, NRF-2022R1I1A1A01072052).}


\begin{abstract}
We consider random normal matrix and planar symplectic ensembles, which can be interpreted as two-dimensional Coulomb gases having determinantal and Pfaffian structures, respectively. 
For general radially symmetric potentials, we derive the asymptotic expansions of the log-partition functions up to and including the $O(1)$-terms as the number $N$ of particles increases. 
Notably, our findings stress that the formulas of the $O(\log N)$- and $O(1)$-terms in these expansions depend on the connectivity of the droplet. 
For random normal matrix ensembles, our formulas agree with the predictions proposed by Zabrodin and Wiegmann up to a universal additive constant. 
For planar symplectic ensembles, the expansions contain a new kind of ingredient in the $O(N)$-terms, the logarithmic potential evaluated at the origin in addition to the entropy of the ensembles.  
\end{abstract}

\maketitle

\section{Introduction and main results}
The Coulomb gas ensemble in the complex plane is governed by the law
\begin{equation} \label{Gibbs cplx beta}
 d P_N^{(\beta)}(z_1,\dots,z_N) :=\frac{1}{Z_N^{(\beta)}} \prod_{j>k=1}^N |z_j-z_k|^{\beta} \prod_{j=1}^N e^{-\frac{\beta N}{2} Q(z_j)}\, dA(z_j),
\end{equation}
where $N$ is the number of particles, $\beta$ is the inverse temperature and $dA(z):=d^2z/\pi$ is the area measure.
Here, $Q:\C \to \R$ is called the confining/external potential that satisfies suitable potential theoretic conditions.
We refer to \cite{forrester2010log,Lewin22,Serfaty} and references therein for recent developments of two-dimensional Coulomb gases. 
Contrary to \eqref{Gibbs cplx beta}, the configurational canonical Coulomb gas ensemble in the upper-half plane \cite{kiessling1999note} (cf. \cite[Appendix A]{byun2022spherical}) has an additional complex conjugation symmetry (i.e. the particles come in complex conjugate pairs) and is governed by the law
\begin{equation} \label{Gibbs symplectic beta}
 d \wt{P}_N^{(\beta)}(z_1,\dots,z_N) :=\frac{1}{ \wt{Z}_N^{(\beta)}} \prod_{j>k=1}^N |z_j-z_k|^\beta |z_j-\bar{z}_k|^\beta \prod_{j=1}^N |z_j-\bar{z}_j|^\beta \, e^{-\beta N Q(z_j)}\, dA(z_j). 
\end{equation}
In \eqref{Gibbs cplx beta} and \eqref{Gibbs symplectic beta}, the normalization constants 
\begin{align} \label{Zn normal def}
Z_N^{(\beta)}&:=\int_{\C^N} \prod_{j>k=1}^N |z_j-z_k|^\beta \prod_{j=1}^N e^{-\frac{\beta N}{2} Q(z_j)}\, dA(z_j), 
\\ 
\wt{Z}_N^{(\beta)}&:= \int_{\C^N} \prod_{j>k=1}^N |z_j-z_k|^\beta |z_j-\bar{z}_k|^\beta \prod_{j=1}^N |z_j-\bar{z}_j|^\beta \, e^{-\beta N Q(z_j)}\, dA(z_j) \label{Zn symplectic def}
\end{align}
that make \eqref{Gibbs cplx beta} and \eqref{Gibbs symplectic beta} probability measures are called partition functions.
Furthermore, the logarithm of a partition function (divided by $N^2$) is often called free energy.

For the special value $\beta=2$, \eqref{Gibbs cplx beta} and \eqref{Gibbs symplectic beta} represent joint probability distributions of the random normal matrix and planar symplectic ensembles, respectively. 
In particular, if $Q(z)=|z|^2$, these correspond to the complex and symplectic Ginibre ensembles \cite{ginibre1965statistical}. 
An important feature of this special value $\beta=2$ is that, due to the Vandermonde determinant terms, the ensembles \eqref{Gibbs cplx beta} and \eqref{Gibbs symplectic beta} form \emph{determinantal} and \emph{Pfaffian} point processes in the plane \cite{forrester2010log}, respectively. 
In other words, all their correlation functions can be expressed in terms of the (pre-)kernel of planar (skew-)orthogonal polynomials. 
In the sequel, for $\beta=2$, we omit the superscript $(\beta)$ in \eqref{Zn normal def} and \eqref{Zn symplectic def}, and simply write $Z_N \equiv Z_N^{(2)}$ and $\wt{Z}_N \equiv \wt{Z}_N^{(2)}$. 

We mention that the definition of partition functions \eqref{Zn normal def} and \eqref{Zn symplectic def} is more common in the statistical physics community. 
On the other hand, in the random matrix theory community, another widely used convention for the (canonical) partition functions is
\begin{equation} \label{partition another convention}
\ZZ_N:=\frac{1}{N!} Z_N, \qquad \wt{\ZZ}_N:= \frac{1}{N!} \wt{Z}_N,
\end{equation}
see e.g. \cite[Section 1.4]{forrester2010log}.
The prefactor $1/N!$ in \eqref{partition another convention} allows writing $\ZZ_N$ and $\wt{\ZZ}_N$ in terms of a structured determinant and Pfaffian, respectively.

In this work, we study the asymptotic expansions of $Z_N$ and $\wt{Z}_N$ as $N \to \infty$.

\subsection{Summary of previous results}

Before introducing our results, let us summarize some known results on the asymptotics of $Z_N^{(\beta)}$ for general $\beta$ and $Q$.
Cf. the literature on $\wt{Z}_N^{(\beta)}$ is much more limited. 

\begin{itemize}
  \item \textbf{(Zabrodin-Wiegmann prediction)} In \cite{MR2240466}, it was predicted that the partition function $Z_N^{(\beta)}$ has an asymptotic expansion of the form 
  \begin{equation} \label{ZW formula}
   \log Z_N^{(\beta)} = C_0 N^2+ C_1 N \log N +C_2 N +C_3 \log N + C_4 +O(\frac{1}{N}). 
  \end{equation}
  Furthermore, they proposed explicit formulas for the constants $C_j\equiv C_j(\beta,Q)$ $(j=0,\dots,4)$ depending on $\beta$ and $Q$, cf. \eqref{ZW formula radial}. 
  (See also \cite{can2015exact,jancovici1994coulomb} for a similar prediction in a different setup.)
  Incidentally, the formulas for $C_3$ and $C_4$ in \cite{MR2240466} have been controversial as pointed out for instance in \cite{sandier20152d,serfaty2020gaussian}. 
  \smallskip 
  \item \textbf{(Asymptotic of the leading order $O(N^2)$-term)} It was shown in \cite[Theorem 2.11]{HM13} and \cite[Theorem 1.1]{chafai2014first} (among others) that as $N \to \infty$,
  \begin{equation}
  \log Z_N^{(\beta)} = -\frac{\beta}{2}N^2 I_Q[\mu_Q] + o(N^2).
\end{equation} 
Here $\mu_Q$ is Frostman's equilibrium measure \cite{ST97}, a unique probability measure that minimizes the weighted logarithmic energy
\begin{equation} \label{IQ functional}
I_Q[\mu] := \iint_{\C^2} \log|z-w| \,d\mu(z)\,d\mu(w) + \int_{\C} Q \,d\mu. 
\end{equation}
  \item \textbf{(Asymptotic up to the $O(N)$-term)} It was shown by Lebl\'{e} and Serfaty \cite[Corollary 1.1]{MR3735628} that as $N \to \infty$, 
  \begin{equation} \label{LS formula}
  \log Z_N^{(\beta)} = -\frac{\beta}{2}N^2 I_Q[\mu_Q] + \frac{\beta}{4}N \log N - \Big( C(\beta) + \Big(1-\frac{\beta}{4}\Big) E_Q[\mu_Q] \Big) N+ o(N),
\end{equation}
where $C(\beta)$ is a constant independent of the potential $Q$ and
\begin{equation}\label{entropy} 
  E_Q[\mu_Q] := \int_{\C} \log (\Lap Q) \, d\mu_{Q}
\end{equation} 
is the entropy associated with $\mu_Q$. Here, $\Lap:=\partial \bar{\partial}$ is the quarter of the usual Laplacian. 
We refer the reader to \cite{bauerschmidt2016two,serfaty2020gaussian} for the expansion \eqref{LS formula} with quantitative error bounds.
\end{itemize}

Beyond the general cases mentioned above, for $\beta=2$ with a specific (and fundamental from the random matrix theory viewpoint) potential, there have been several works on the precise asymptotic expansion of the partition functions, see e.g. \cite{charlier2021large,charlier2021asymptotics} and references therein. 
This type of potential usually contains certain singularities. As a result, the asymptotic expansions of the associated partition functions are more complicated (for instance, some non-trivial $O(\sqrt{N})$ terms appear as well).
Several topics in this direction will be discussed in a separate remark at the end of the next subsection.

\subsection{Main results}

We study asymptotic behaviors of $Z_N$ and $\wt{Z}_N$ for the exactly-solvable case where $Q$ is radially symmetric. 
Our main findings are summarized as follows.
\begin{enumerate}[label=(\roman*)]
  \item We derive the large-$N$ expansions of $\log Z_N$ and $\log \wt{Z}_N$ up to and including the $O(1)$-terms. 
  \smallskip 
  \item In the large-$N$ expansions, the formulas of the $O(\log N)$- and $O(1)$-terms depend on whether the limiting spectrum is an \emph{annulus} or a \emph{disc}, see Theorems~\ref{Thm_ZN annulus} and \ref{Thm_ZN disc}, respectively, cf. Figure~\ref{fig:droplets}. 
  (This distinction is crucial in the asymptotic analysis but seems not considered in \cite{MR2240466}.)
  \smallskip 
  \item For the partition function $Z_N$ of random normal matrix ensembles, our expansions \eqref{ZN expansion cplx annulus} and \eqref{ZN expansion cplx disc} up to the $O(N)$-terms agree with the formula \eqref{LS formula} with $\beta=2$.
  Furthermore, we verify from \eqref{ZN expansion cplx disc} that the asymptotic formula given in \cite[Eqs.(1.2), (C.7)]{MR2240466} holds up to an additive constant \eqref{ZW missing}. 
  \smallskip 
  \item For the partition function $\wt{Z}_N$ of planar symplectic ensembles, the asymptotic formulas \eqref{ZN expansion symp annulus} and \eqref{ZN expansion symp disc} are new to the best of our knowledge. 
  Contrary to \eqref{LS formula}, the $O(N)$-terms in these expansions contain not only the entropy but also the logarithmic potential \eqref{logarithmic potential}. 
\end{enumerate}

\begin{figure}[h!]
\begin{center}
\centering
\includegraphics[width=0.8\textwidth]{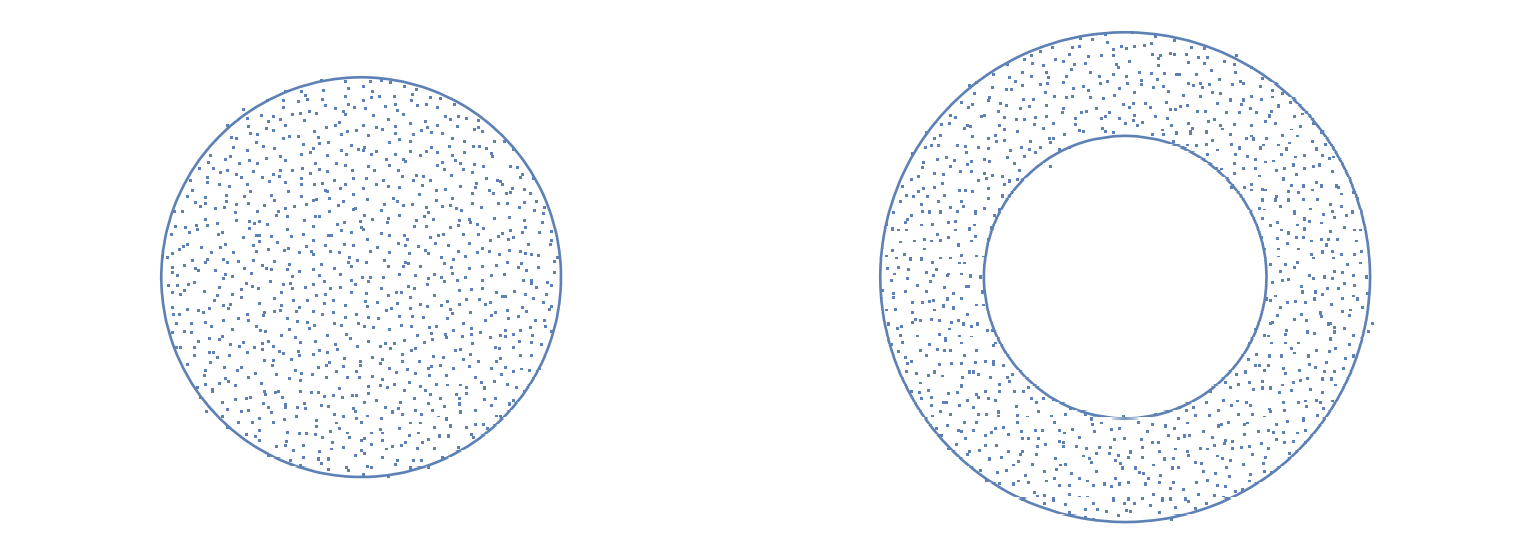}
\caption{Eigenvalues of complex Ginibre (left) and complex induced Ginibre (right) matrices where $N=1000$.} 
\label{fig:droplets}
\end{center}
\end{figure}

Let us be more precise in introducing our results. 
It is well known \cite{MR2934715,HM13} that under some mild assumptions on $Q$, as $N \to \infty$, the empirical measures $\frac{1}{N}\sum_{j=1}^N \delta_{z_j}$ of \eqref{Gibbs cplx beta} and \eqref{Gibbs symplectic beta} weakly converge to $\mu_Q$, which takes the form
\begin{equation} \label{eq msr Lap Q}
d\mu_Q =\Delta Q \cdot \mathbbm{1}_S \,dA.
\end{equation}
Here $S \equiv S_Q$ is a certain compact subset of $\C$ called the \emph{droplet}, see Figure~\ref{fig:droplets}. 

We consider the case where the external potential $Q$ is radially symmetric, i.e. $Q(z) = q(|z|)$ for some function $q$ defined in $[0,\infty)$. 
We assume the basic growth condition 
\begin{equation} \label{Q growth condition}
\liminf_{ |z| \to \infty } \frac{Q(z)}{ 2\log |z| } >1,  
\end{equation}
which guarantees that $Z_N, \wt{Z}_N < +\infty.$
Furthermore, we assume that $Q$ is smooth, subharmonic in $\C$, and strictly subharmonic in a neighborhood of the droplet. 
In terms of the function $q,$ the latter conditions are equivalent to the requirements that $r q'(r)$ is increasing on $(0,\infty)$, and strictly increasing in a neighborhood of the droplet, cf. \eqref{Delta Q q' q''}. 
Under the above assumptions, the droplet is given by 
\begin{equation} \label{droplet}
S = \mathbb{A}_{r_0,r_1}:= \{z \in \C : r_0 \leq |z| \leq r_1\},
\end{equation}
where $r_0$ is the largest solution to $r q'(r)=0$ and $r_1$ is the smallest solution to $r q'(r)=2$, see \cite[Section IV.6]{ST97}.
(We mention that the annular droplets often appear in non-Hermitian random matrix theory, see e.g. \cite{MR2831116}.)
In particular, if $r_0=0$, we denote $\mathbb{D}_{r_1}= \mathbb{A}_{0,r_1}.$
Henceforth, we keep the assumptions on $Q$ described above. 

For a radially symmetric potential $Q$, by using \eqref{eq msr Lap Q} and \eqref{droplet}, one can show that the energy $I_Q[\mu_Q]$ in \eqref{IQ functional} is given by 
\begin{equation} \label{energy radially sym}
I_Q[\mu_Q]=q(r_1)-\log r_1 -\frac14 \int_{r_0}^{r_1} rq'(r)^2\,dr. 
\end{equation}
Similarly, in terms of the logarithmic potential 
\begin{equation} \label{logarithmic potential}
U_\mu(z) = \int \log\frac1{|z-w|}\, d\mu(w),
\end{equation}
we have 
\begin{equation} \label{potential radially sym}
U_{\mu_Q}(0) = -\int_{S}\log |w| \, d\mu_Q(w) = - \log r_1 + \frac{q(r_1)-q(r_0)}{2}. 
\end{equation}
See \cite[Section IV.6]{ST97} for more details. 

\medskip 

For the annular droplet case, we have the following.

\begin{thm} \textbf{\textup{(Large-$N$ expansion of the partition functions: annular droplet case)}} \label{Thm_ZN annulus}
Suppose that $r_0>0$, i.e. the droplet $S$ in \eqref{droplet} is an annulus. 
Let 
\begin{equation} \label{FQ annulus}
 F_Q[\mathbb{A}_{r_0,r_1}] := \frac{1}{12} \log\Big( \frac{ r_0^2\Delta Q(r_0) }{ r_1^2\Delta Q(r_1) } \Big) -\frac{1}{16} \Big( r_1 \frac{ (\d_r \Delta Q)(r_1) }{ \Delta Q(r_1) } - r_0 \frac{ (\d_r \Delta Q)(r_0) }{ \Delta Q(r_0) } \Big) +\frac{1}{24} \int_{r_0}^{r_1} \Big( \frac{ \d_r \Delta Q(r) }{ \Delta Q(r) }\Big)^2\,r \,dr. 
\end{equation}
Then as $N \to \infty,$ the following holds. 
\begin{enumerate}[label=(\roman*)]
  \item \textup{\textbf{(Random normal matrix ensemble)}} We have
\begin{equation} \label{ZN expansion cplx annulus}
\begin{split} 
\log Z_N &= -N^2 I_Q[\mu_Q] + \frac{1}{2}N\log N + \Big( \frac{\log(2\pi)}{2}-1- \frac12 E_Q[\mu_Q] \Big) \, N 
\\
&\quad + \frac{1}{2}\log N + \frac{\log (2\pi)}{2} +F_Q[\mathbb{A}_{r_0,r_1}] + O(N^{-1}).
\end{split}
\end{equation}
  \item \textup{\textbf{(Planar symplectic ensemble)}} We have 
\begin{equation}
\begin{split} \label{ZN expansion symp annulus}
 \log \wt{Z}_N & = -2N^2 I_Q[\mu_Q] +\frac{1}{2}N \log N + \Big( \frac{ \log(4\pi) }{2}-1 - U_{\mu_Q}(0) -\frac12 E_Q[\mu_Q] \Big) N  
\\
&\quad+\frac12 \log N + \frac{\log(2\pi)}{2} +\frac12 F_Q[\mathbb{A}_{r_0,r_1}]+\frac{1}{8} \log \Big( \frac{ \Delta Q(r_0) }{ \Delta Q(r_1) } \Big) +O(N^{-1}).
\end{split}
\end{equation}
\end{enumerate}
\end{thm}

Using the convention \eqref{partition another convention} together with \eqref{log N!}, our result can also be rewritten as 
\begin{equation} \label{ZN expansion cplx annulus v2}
\begin{split} 
\log \ZZ_N &=-N^2 I_Q[\mu_Q] - \frac{1}{2}N\log N + \Big( \frac{\log(2\pi)}{2}- \frac12 E_Q[\mu_Q] \Big) \, N +F_Q[\mathbb{A}_{r_0,r_1}] + O(N^{-1}).
\end{split}
\end{equation}
and 
\begin{equation}
\begin{split} \label{ZN expansion symp annulus v2}
 \log \wt{\ZZ}_N & = -2N^2 I_Q[\mu_Q] - \frac{1}{2}N\log N + \Big( \frac{ \log(4\pi) }{2} - U_{\mu_Q}(0) -\frac12 E_Q[\mu_Q] \Big) N 
 \\
 &\quad +\frac12 F_Q[\mathbb{A}_{r_0,r_1}]+\frac{1}{8} \log \Big( \frac{ \Delta Q(r_0) }{ \Delta Q(r_1) } \Big) +O(N^{-1}).
\end{split}
\end{equation}
We mention that these formulas \eqref{ZN expansion cplx annulus v2} and \eqref{ZN expansion symp annulus v2} as well as the formulas \eqref{ZN expansion cplx disc v2} and \eqref{ZN expansion symp disc v2} below are more convenient to compare with some asymptotic results in the previous literature \cite{charlier2021large,charlier2021asymptotics}.

Notice that the term $\log (r_1/r_0)$ is the extremal length of the annulus \eqref{droplet}, see e.g. \cite[p.142]{garnett2005harmonic}. 
It is worth pointing out that a characteristic difference between the expansions \eqref{ZN expansion cplx annulus} and \eqref{ZN expansion symp annulus} is the appearance of the logarithmic potential $U_{\mu_Q}(0)$ in the $O(N)$-term of \eqref{ZN expansion symp annulus}. 
This additional term can be interpreted as
\begin{equation} \label{renormalized energy}
- \int_S \log|w-\bar{w}| \,d\mu_Q(w)
\end{equation}
after a proper renormalization. 
The interpretation \eqref{renormalized energy} is natural from the perspective of the repulsion term $ |z_j-\bar{z}_j|^\beta $ in \eqref{Gibbs symplectic beta} and is closely related to the notion of the next-order energy, see e.g. \cite{leble2018fluctuations}. 
(We thank T. Lebl\'{e} for pointing out this.)

In Subsection~\ref{Subsec_ML} we present an example of Theorem~\ref{Thm_ZN annulus} for the Mittag-Leffler ensembles from which we expect that the error terms $O(N^{-1})$ are optimal.

\medskip 

For the disc droplet case, we have the following.

\begin{thm} \textbf{\textup{(Large-$N$ expansion of the partition functions: disc droplet case)}} \label{Thm_ZN disc}
Suppose that $r_0=0$, i.e. the droplet $S$ in \eqref{droplet} is a disc. 
Let 
\begin{equation}
\begin{split} \label{FQ disc}
 F_Q[\mathbb{D}_{r_1}] := \frac{1}{12} \log\Big( \frac{ 1 }{ r_1^2 \Delta Q(r_1) } \Big) 
 -\frac{1}{16} r_1 \frac{ (\d_r \Delta Q)(r_1) }{ \Delta Q(r_1) } +\frac{1}{24} \int_{0}^{r_1} \Big( \frac{ \d_r \Delta Q(r) }{ \Delta Q(r) }\Big)^2\,r \,dr. 
\end{split}
\end{equation}
Then as $N \to \infty,$ the following holds. 
\begin{enumerate}[label=(\roman*)]
  \item \textup{\textbf{(Random normal matrix ensemble)}} We have
\begin{align}
\begin{split} \label{ZN expansion cplx disc}
  \log Z_N 
  &= -N^2 I_Q[\mu_Q] + \frac{1}{2}N\log N + \Big( \frac{\log(2\pi)}{2}-1- \frac12 E_Q[\mu_Q] \Big) \, N + \frac{5}{12}\log N 
  \\
  &\quad + \frac{\log (2\pi)}{2} + \zeta'(-1)+ F_Q[\mathbb{D}_{r_1}] + O(N^{-\frac{1}{12}}(
\log N)^{3}).
\end{split}
\end{align}
  \item \textup{\textbf{(Planar symplectic ensemble)}} We have 
\begin{align}
\begin{split}
\label{ZN expansion symp disc}
 \log \wt{Z}_N & = -2N^2 I_Q[\mu_Q] +\frac{1}{2}N \log N + \Big( \frac{ \log(4\pi) }{2}-1 - U_{\mu_Q}(0) -\frac12 E_Q[\mu_Q] \Big) N+\frac{11}{24} \log N \\
&\quad + \frac{\log(2\pi)}{2}+\frac{1}{2}\zeta'(-1) +\frac12 F_Q[\mathbb{D}_{r_1}]+\frac{5}{24}\log2+ \frac{1}{8} \log \Big( \frac{ \Delta Q(0) }{ \Delta Q(r_1) } \Big) + O(N^{-\frac{1}{12}}
(\log N)^3).
\end{split}
\end{align}
\end{enumerate}
Here $\zeta$ is the Riemann zeta function.
\end{thm}

Again, using the convention \eqref{partition another convention}, we have
\begin{align}
\begin{split} \label{ZN expansion cplx disc v2}
  \log \ZZ_N
  &= -N^2 I_Q[\mu_Q] - \frac{1}{2}N\log N + \Big( \frac{\log(2\pi)}{2}- \frac12 E_Q[\mu_Q] \Big) \, N - \frac{1}{12}\log N 
  \\
   &\quad + \zeta'(-1)+ F_Q[\mathbb{D}_{r_1}] + O(N^{-\frac{1}{12}}( \log N)^{3}) \end{split} 
\end{align}
and 
\begin{align}
\begin{split}
\label{ZN expansion symp disc v2}
 \log \wt{\ZZ}_N & = -2N^2 I_Q[\mu_Q] - \frac{1}{2}N \log N + \Big( \frac{ \log(4\pi) }{2} - U_{\mu_Q}(0) -\frac12 E_Q[\mu_Q] \Big) N-\frac{1}{24} \log N \\ &\quad +\frac{1}{2}\zeta'(-1) +\frac12 F_Q[\mathbb{D}_{r_1}]+\frac{5}{24}\log2+ \frac{1}{8} \log \Big( \frac{ \Delta Q(0) }{ \Delta Q(r_1) } \Big) + O(N^{-\frac{1}{12}} (\log N)^3). 
\end{split} 
\end{align}

In Subsection~\ref{Subsec_truncated}, we provide an example of Theorem~\ref{Thm_ZN disc} for truncated unitary ensembles.
Contrary to Theorem~\ref{Thm_ZN annulus}, the error terms in Theorem~\ref{Thm_ZN disc} do not coincide with the expected optimal orders $O(N^{-1})$. 
Our error bounds originate from a decomposition of the analytic expressions of $Z_N,\wt{Z}_N$ (see Subsection~\ref{Subsec_outline}), which depends on sufficiently large but seemingly arbitrary number $m_N>0.$ 
(Such a decomposition was not necessary for the proof of Theorem~\ref{Thm_ZN annulus}.)
Later, we choose $m_N=N^{1/6}$ that gives rise to the control of the total error bounds presented in Theorem~\ref{Thm_ZN disc}. 
We mention that such error estimates also naturally appeared in similar computations, see e.g. \cite{byun2022characteristic,charlier2021large,charlier2021asymptotics}. 
Nevertheless, we expect that the estimates can be improved with more effort.

In terms of the function $\chi:= \frac12 \log \Lap Q$, one can rewrite \eqref{FQ disc} as 
\begin{equation} \label{FQ disc v2}
F_Q[\mathbb{D}_{r_1}] =\frac{1}{12}\log\Big( \frac{1}{r_1^2}\Big) -\frac{1}{12\pi} \oint_{ \partial S } \kappa \, \chi\,ds -\frac{1}{4} \int_S \Lap \chi \,dA + \frac{1}{12} \int_S | \nabla \chi |^2 \,dA,
\end{equation}
where $S=\mathbb{D}_{r_1}$ and $\kappa=1/r_1$ is the curvature of the boundary, see \cite[p.8960]{MR2240466} and \eqref{ZW F0 12 1}.
Here, the third term $\int_S \Lap \chi \,dA$ on the right-hand side of \eqref{FQ disc v2} is known as a ``zero mode'' of the loop operator (cf. \cite[Eq.(5.26)]{MR2240466}), whereas the fourth term corresponds to the Dirichlet energy of $\chi$.

We end this subsection by giving some crucial remarks on our theorems. 

\begin{rem}[\textbf{Comparison with Zabrodin-Wiegmann formula}]
We compare our formula \eqref{ZN expansion cplx disc} with the prediction by Zabrodin and Wiegmann. 
For $\beta=2$ and a radially symmetric $Q$ associated with a disc droplet of radius $r_1$, the asymptotic formula \eqref{ZW formula} is written in \cite[Eqs.(1.2),(C.7)]{MR2240466} as 
\begin{equation}
\begin{split} \label{ZW formula radial}
\log Z_N &=F_0 N^2 + F_{1/2} N +F_1 +c(N) +O(N^{-1}), \qquad \Big( c(N):=\log N! -\frac{1}{2}N \log N + \gamma N \Big)
\\
&= F_0 N^2+ \frac12 N \log N + \Big( \gamma-1+F_{1/2} \Big) N +\frac12 \log N + \frac{ \log(2\pi) }{ 2 } +F_1 +O(N^{-1}),
\end{split}
\end{equation}
where $\gamma$ is a ``numerical'' constant \cite[p.8938]{MR2240466} (that is not explicitly presented). 
(For reader's convenience, let us mention that in \cite{MR2240466}, the authors use a different convention for $\beta$ so that $\beta=2$ in our case corresponds to $\beta=1$ in \cite{MR2240466}. Furthermore, the Planck constant $\hbar$ in \cite{MR2240466} is identified as $1/N$.)
The coefficients $F_0, F_{1/2}$ and $F_1$ in \eqref{ZW formula radial} are given by
\begin{align}
F_0&:= \pi \int_0^{r_1^2} ( W_{ \textrm{rad} }(x)-W_{ \textrm{rad} }'(x) x \log x ) \sigma_{ \textrm{rad} }(x) \,dx, \label{F0 ZW}
\\
F_{1/2} &:= -\frac{\pi}{2} \int_0^{r_1^2} \sigma_{ \textrm{rad} }(x) \log (\pi \sigma_{ \textrm{rad} }(x))\,dx, \label{F12 ZW}
\\
F_1&:= \frac{1}{12} \log\Big( \frac{1}{r_1^2} \Big)- \frac{1}{6} \chi_{ \textrm{rad} }(r_1^2) - \frac14 r_1^2 \chi_{ \textrm{rad} }'(r_1^2)+\frac13 \int_0^{r_1^2} x (\chi_{ \textrm{rad} }'(x))^2\,dx, \label{F1 ZW}
\end{align}
where 
\begin{equation} \label{chi rad def}
W_{ \textrm{rad} }(r):= -Q(\sqrt{r}) , \qquad 
\sigma_{ \textrm{rad} }(r):=\frac{1}{\pi} \Lap Q(\sqrt{r}), \qquad 
\chi_{ \textrm{rad} }(r):= \frac12 \log \Lap Q(\sqrt{r}). 
\end{equation}
Using \eqref{chi rad def}, it is straightforward to check that the formulas \eqref{F0 ZW}, \eqref{F12 ZW} and \eqref{F1 ZW} can be identified as
\begin{equation} \label{ZW F0 12 1}
F_0=-I_Q[\mu_Q],\qquad F_{1/2}=-\frac12 E_Q[\mu_Q], \qquad F_1=F_Q[\mathbb{D}_{r_1}],
\end{equation}
where $I_Q$, $E_Q$ and $F_Q[\mathbb{D}_{r_1}]$ are given by \eqref{energy radially sym}, \eqref{entropy} and \eqref{FQ disc}.
(Cf. the identification of $F_0$ follows from the computation \eqref{sum Vtau EM 1} below.)
Then by letting $\gamma=\log(2\pi)/2$, one can deduce from \eqref{ZW F0 12 1} that the asymptotic formula \eqref{ZW formula radial} agrees with our result \eqref{ZN expansion cplx disc} up to the additive terms
\begin{equation} \label{ZW missing}
-\frac1{12} \log N +\zeta'(-1).
\end{equation}

We remark that the asymptotic expansion of the partition function of the complex Ginibre ensemble was presented in \cite{can2015exact,tellez1999exact}, where one can also observe the term $\zeta'(-1)$ in \eqref{ZW missing}.
Indeed, the term $\zeta'(-1)$ and its generalizations have appeared in similar situations in the Hermitian matrix theory, see \cite[Remark 1.3, Proposition 1.4]{MR4259377}, \cite[Theorem 1.1]{MR4043828}, \cite[Theorem 1.2]{MR4232544}. Interestingly, the coefficients of the $\zeta'(-1)$ term depend on the connectivity of the droplet (i.e. the number of disjoint intervals in this case) and the number of hard edges. 
\end{rem}

\begin{rem}[\textbf{Non-triviality of the limit $r_0 \to 0$}]
The formulas \eqref{ZN expansion cplx disc} and \eqref{ZN expansion symp disc} cannot be recovered by simply taking the limit $r_0 \to 0$ of \eqref{ZN expansion cplx annulus} and \eqref{ZN expansion symp annulus}. 
Namely, it is obvious that as $r_0 \to 0$, the terms
\begin{equation} \label{ZW missing annulus}
-\frac1{12} \log\Big( \frac{1}{r_0^2 \Lap Q(r_0) } \Big)+\frac{1}{16} r_0 \frac{ (\d_r \Delta Q)(r_0) }{ \Delta Q(r_0) } 
\end{equation}
do not correspond to the terms \eqref{ZW missing}. 
(One may however notice that for the standard microscale $r_0=O(1/\sqrt{N})$, at least the $-\frac1{12}\log N$ term in \eqref{ZW missing} follows.)

From the viewpoint of the proof, the origin of \eqref{ZW missing} and \eqref{ZW missing annulus} is essentially similar in the sense that these terms arise from the asymptotic behaviors of the summand in \eqref{ZN random normal symplectic} of \emph{lower degrees}.
Nevertheless, it is essential (but seems not discussed in \cite{MR2240466}) that these asymptotic behaviors depend on whether the droplet is contractible or not, i.e. for the radially symmetric potentials, disc or annulus. 
We remark that the contractible case requires considerably more analysis than the other case, see the following subsection for more discussion.
\end{rem}

\begin{rem}[\textbf{Invariance of the $O(1)$-terms under the dilation}]
For $a>0$, let $Q_a(z):=Q(z/a)$. Then the droplet associated with $Q_a$ is given by $\{ z\in \C: ar_0 \le |z| \le ar_1 \}$, where $r_0$ and $r_1$ are given in \eqref{droplet}. 
Then it follows from \eqref{FQ annulus} and \eqref{FQ disc} that
\begin{equation}\label{FQ invariance dilation}
F_Q[\mathbb{A}_{r_0,r_1}]=F_{Q_a}[\mathbb{A}_{ar_0,ar_1}], \qquad F_Q[\mathbb{D}_{r_1}]= F_{Q_a}[\mathbb{D}_{ar_1}]. 
\end{equation}
This in turn means that the $O(1)$-terms in the expansions in Theorems~\ref{Thm_ZN annulus} and \ref{Thm_ZN disc} are invariant under the dilation $\{ z_j \} \mapsto \{a\cdot z_j \}$.
The property \eqref{FQ invariance dilation} can be expected from the analytic expression \eqref{ZN random normal symplectic} below. 
\end{rem}

\begin{rem}[\textbf{Weight function with singularities and classical problems in random matrix theory}]
In Theorems~\ref{Thm_ZN annulus} and \ref{Thm_ZN disc}, we focus on the weight function $e^{-NQ}$ without any kind of singularities. 
In contrast, if a specific singularity is allowed for the weight function, the problems of deriving asymptotic expansions of the associated partition function are (when combined with Theorems~\ref{Thm_ZN annulus} and \ref{Thm_ZN disc}) equivalent to several classical problems in random matrix theory.

To be more concrete, we list various problems in this direction. 
If the weight function has a \emph{hard-edge} inside the droplet, the associated partition function provides the large gap (hole) probability, see \cite{adhikari2018hole,AR2016hole,akemann2009gap,byun2022almost,charlier2021large,MR1181356,ghosh2018point,jancovici1993large} and references therein.
The weight function with a \emph{jump-type singularity} gives rise to the moment generating function of the disc counting function. It has been extensively studied in recent years \cite{akemann2022universality,ameur2022exponential,byun2022characteristic,charles2020entanglement,charlier2021asymptotics,charlier2022exponential}. 
(We also refer to \cite{fenzl2022precise,lacroix2019intermediate,lacroix2019rotating,smith2021counting}
for physical motivations of these problems from the counting statistics of rotating free fermions.) 
Finally, a \emph{root-type singularity} arises in the study of the log-characteristic polynomials \cite{byun2022characteristic,deano2019characteristic,MR3946715}. 

We stress that the literature mentioned above is limited mainly to a particular model, such as the Ginibre ensemble, when deriving precise asymptotic results or to the leading order asymptotic when considering general potentials.
We expect that Theorems~\ref{Thm_ZN annulus} and \ref{Thm_ZN disc} provide the building blocks for obtaining precise asymptotic results on the problems mentioned above with general radially symmetric potentials. 
\end{rem}

\begin{rem}[\textbf{Planar point processes with a general external potential $Q$}]
For a general potential $Q$ beyond a radially symmetric one, the asymptotic behaviors of planar orthogonal polynomials (of sufficiently large degrees) with respect to $e^{-NQ}\,dA$ were recently obtained in \cite{hedenmalm2017planar}.
We expect that this will be helpful to extend Theorem~\ref{Thm_ZN annulus} (i) to a general potential $Q$ associated with a ``non-contractible'' droplet. 
On the other hand, for the extension of Theorem~\ref{Thm_ZN disc} (i), it is required to derive asymptotics of orthogonal polynomials of lower degrees as well.

Such a generalization of Theorems~\ref{Thm_ZN annulus} and \ref{Thm_ZN disc} (ii) for planar symplectic ensembles seems at present far from being solved. 
More precisely, in order to obtain an analytic expression of $\wt{Z}_N$, it is required to construct the associated skew-orthogonal polynomial. 
However, for a non-radially symmetric potential, this construction has been known only in a few special cases \cite{MR2180006, MR1928853} (cf. see \cite{akemann2021skew} for a possible generality).
\end{rem}

\subsection{Outline of the proof} \label{Subsec_outline}
In this subsection, we outline the proofs of our main results.
Using the determinantal (resp., Pfaffian) structure and de Bruijn's type formulas, one can express $Z_N$ (resp., $\wt{Z}_N$) in terms of the (skew-)orthogonal norms.
Consequently, since $Q$ is radially symmetric, we find
\begin{equation} \label{ZN random normal symplectic}
\log Z_N=\log N!+ \sum_{j=0}^{N-1} \log h_j, \qquad \log \wt{Z}_N= \log N! + \sum_{j=0}^{N-1} \log (2\widetilde{h}_{2j+1}) ,
\end{equation}
where
\begin{equation} \label{hj hj tilde def}
h_j:= \int_\C |z|^{2j}\,e^{-N Q(z)}\,dA(z), \qquad \widetilde{h}_j:= \int_\C |z|^{2j}\,e^{-2N Q(z)}\,dA(z).
\end{equation}
These formulas can be found for instance in \cite[Lemma 1.7]{charlier2021asymptotics} and \cite[Remark 2.5]{akemann2021skew}. 
In particular, for planar symplectic ensembles, we have used the explicit construction of skew-orthogonal polynomials associated with radially symmetric potentials, see \cite[Corollary 3.3]{akemann2021skew}. 

In order to obtain the large-$N$ expansions of partition functions up to the $O(1)$-terms, we need to derive asymptotic behaviors of $h_j$ and $\widetilde{h}_j$ up to the first subleading terms, for which we apply Laplace's method. 
For this purpose, let $r_\tau$ be a unique number $r_\tau$ such that $r_\tau q'(r_\tau)=2\tau$ for $0\le \tau \le 1$.
Such a function $\tau \mapsto r_\tau$ plays an important role in Laplace's method, and we defer more explanations to Subsection~\ref{Subsection_norms annulus}.
In the asymptotic expansions of $h_j$ and $\widetilde{h}_j$, one should distinguish the following two cases depending on a small constant $\varepsilon>0$. 
\begin{itemize}
  \item \textbf{Case 1: $r_{j/N} \gg N^{-\varepsilon}$.} For the annular droplet case where $r_0>0$, this case covers all $j=0,1,\dots,N-1$ (Lemma~\ref{Lem_hj asymptotic}). 
  On the other hand, for the disc droplet case where $r_0=0$, this case covers only $j=m_N,m_N+1,\dots,N-1$ for $m_N=N^{\epsilon}$ with some $\epsilon>0$ (Lemma~\ref{Lem_hj higher}). 
  \smallskip 
  \item \textbf{Case 2: $r_{j/N} \ll N^{-\varepsilon}$.} This covers the remaining disc droplet case with $j=0,1,\dots,m_N-1$ (Lemma~\ref{Lem_hj lower}). Notably, the asymptotic expansion involves gamma functions in this case. 
\end{itemize}

Furthermore, we apply the Euler-Maclaurin formula (see e.g. \cite[Section 2.10]{olver2010nist}) to precisely analyze the summations in \eqref{ZN random normal symplectic}
\begin{equation}
\begin{split} \label{EMF}
\sum_{j=m}^{n} f(j) = \int_{m}^{n} f(x) \, dx + \frac{f(m)+f(n)}{2} + \sum_{k=1}^{l} \frac{B_{2k}}{(2k)!}\Big(f^{(2k-1)}(n)-f^{(2k-1)}(m)\Big) + R_l,
\end{split}
\end{equation}
where $B_k$ is the Bernoulli number defined by 
\begin{equation} \label{Bernoulli number}
\frac{t}{e^t-1}=\sum_{n=0}^\infty B_n \frac{t^n}{n!}
\end{equation}
and $R_l$ is an error term with the following estimate 
\[
|R_l| \le \frac{2\zeta(l)}{(2\pi)^l} \int_m^n|f^{(l)}(x)|\,dx.
\]
Here $\zeta$ is the Riemann zeta function.
In particular, for the disc droplet case, in the summation of lower degrees $j=0,1,\dots, m_N-1$, we consider the Barnes $G$-function \cite[Section 5.17]{olver2010nist}
\begin{equation*}
  G(z+1) = (2\pi)^{z/2} e^{-(z+z^2(1+\gamma))/2} \prod_{n=0}^\infty \big(1+\frac z n\big)^n e^{-z+z^2/(2n)}
\end{equation*}
and use its asymptotic expansion \cite[Eq.(5.17.5)]{olver2010nist}: 
\begin{align}
\begin{split} \label{Barnes G asymp}
\log G(z+1) =\frac{z^2 \log z}{2} -\frac34 z^2+\frac{ \log(2\pi) z}{2}-\frac{\log z}{12}+\zeta'(-1)+O( \frac{1}{z^2} ), \qquad (z \to \infty).
\end{split}
\end{align}
This asymptotic expansion leads to the appearance of the Riemann zeta function in Theorem~\ref{Thm_ZN disc}.

\subsection*{Plan of the paper} The rest of this paper is organized as follows. 
In Section~\ref{Section_annulus}, we prove Theorem~\ref{Thm_ZN annulus}.
Subsection~\ref{Subsection_norms annulus} is devoted to deriving asymptotic behaviors of $h_j$ and $\widetilde{h}_j$ using Laplace's method. 
Then we show Theorem~\ref{Thm_ZN annulus} (i) on random normal matrices in Subsection~\ref{Subsection_annulus normal} and Theorem~\ref{Thm_ZN annulus} (ii) on planar symplectic ensembles in Subsection~\ref{Subsection_annulus symplectic}.
Section~\ref{Section_disc} is structured in parallel with a goal to show Theorem~\ref{Thm_ZN disc} albeit it requires considerably more computations compared to those in Section~\ref{Section_annulus}.
In Section~\ref{Appendix_Examples}, we present examples of Theorems~\ref{Thm_ZN annulus} and \ref{Thm_ZN disc} for the Mittag-Leffler and truncated unitary ensembles whose partition functions can be explicitly expressed in terms of well-known special functions.

\section{Proof of Theorem~\ref{Thm_ZN annulus}} \label{Section_annulus}

In this section, we prove Theorem~\ref{Thm_ZN annulus}. 
Throughout this section, we assume that $r_0>0.$

\subsection{Asymptotics of the orthogonal norm} \label{Subsection_norms annulus}
We first introduce an auxiliary function $V_\tau$ in $(0,\infty)$ 
\begin{equation} \label{V tau}
  V_\tau(r) := q(r) - 2\tau \log r.
\end{equation}
With the following choices of $\tau = \tau(j), \wt\tau(j)$
\begin{equation} \label{tau j wt tau j}
\tau(j):=\frac{j}{N}, \qquad \wt{\tau}(j):=\frac{j}{2N},
\end{equation}
the integrands in $h_j, \wt{h}_j$ \eqref{hj hj tilde def} can be expressed in terms of $V_\tau:$
$$r^{2j} e^{-N q(r)} = e^{-N V_{\tau(j)}(r)}, \qquad r^{2j} e^{-2N q(r)} = e^{-N V_{\wt{\tau}(j)}(r)}.$$
For a radially symmetric potential $Q$, we represent $\Delta Q$ in terms of $q$ as 
\begin{equation} \label{Delta Q q' q''}
4\Delta Q(z)|_{z=r}= \frac{1}{r}(rq'(r))'= \frac{q'(r)}{r}+q''(r) . 
\end{equation}
Differentiating \eqref{V tau}, we have 
\begin{align}
\begin{split} \label{V tau diff}
V_\tau'(r)=q'(r)-\frac{2\tau}{r}, \qquad V_\tau''(r) &= 4\Delta Q(r)-\frac{1}{r} V_\tau'(r), \qquad 
V_\tau^{(3)}(r) = 4 \d_r \Delta Q(r)-\frac{4}{r} \Delta Q(r) +\frac{2}{r^2} V_\tau'(r), \\
V_\tau^{(4)}(r) &= 4 \d_r^2 \Delta Q(r)+\frac{12}{r^2} \Delta Q(r)-\frac{4}{r} \d_r \Delta Q(r) -\frac{6}{r^3} V_\tau'(r).
\end{split}
\end{align}

We now set the stage to apply Laplace's method.
Since $rq'(r)$ is strictly increasing inside the droplet, for $0 \leq \tau \leq 1$, there exists a unique number $r(\tau)$ such that 
\begin{equation} \label{r tau min}
V_\tau'(r(\tau))=0 , \qquad V_\tau''(r(\tau))>0.
\end{equation}
Moreover, by \eqref{Delta Q q' q''} and the relation \eqref{r tau min}, it follows that 
\begin{equation}\label{drdtau}
  \frac{dr(\tau)}{d\tau} = \frac{2}{(rq'(r))'}\Big|_{r=r(\tau)} = \frac{1}{2r(\tau)\Lap Q (r(\tau))}>0.
\end{equation}
Thus $r(\tau)$ is an increasing function of $\tau$. 
On the other hand, $r(1) = r_1, r(0) = r_0$, where $r_0,r_1$ are given in \eqref{droplet}. Therefore, we denote $r_\tau = r(\tau)$ making the notation consistent with \eqref{droplet}. 
We also mention here that $r_\tau$ corresponds to the outer radius of the so-called ``$\tau$-droplet'' \cite{hedenmalm2017planar}. 
By \eqref{V tau diff} and \eqref{r tau min}, $r_\tau$ satisfies
\begin{equation} \label{r tau eq}
r_\tau q'(r_\tau)=2\tau.
\end{equation}
In particular, $q'(r_0)=0$ and $r_1q'(r_1)=2.$

Let
\begin{equation} \label{B1}
\mathfrak{B}_1(r):= -\frac{1}{32} \frac{ \d_r^2 \Delta Q(r) }{ (\Delta Q(r))^2 } -\frac{19}{96r} \frac{\d_r \Delta Q(r)}{ (\Delta Q(r))^2 }+\frac{5}{96} \frac{ (\d_r \Delta Q(r))^2 }{ \Delta Q(r)^3 } +\frac{1}{12r^2} \frac{1}{\Delta Q(r)}.
\end{equation}
Here, the subscript $1$ is added to $\mathfrak{B}$ to emphasize that this function appears as the first subleading term of the asymptotic expansion of orthonormal polynomials, see Lemma~\ref{Lem_hj asymptotic} below.
Indeed, function $\mathfrak{B}_1$ is closely related to function $\mathfrak{B}_{\tau,1}$ in \cite[Theorem 1.3]{hedenmalm2017planar}. 

\begin{lem}\label{Lem_hj asymptotic} 
As $N\to\infty$, the following holds.
\begin{itemize}
  \item For each $j$ with $0\leq j \leq N-1$, 
\begin{equation}
  h_j = N^{-\frac{1}{2}}e^{-NV_{\tau(j)}(r_{\tau(j)})} \Big(\frac{{2\pi}r_{\tau(j)}^2}{\Lap Q (r_{\tau(j)})}\Big)^{\frac{1}{2}} \Big( 1 + \frac{1}{N}\mathfrak{B}_1(r_{\tau(j)}) + O(N^{-2})\Big).
\end{equation}
\item For each $j$ with $0\leq j \leq 2N-1$, 
\begin{equation}
  \widetilde{h}_j = (2N)^{-\frac{1}{2}}e^{-2NV_{\widetilde{\tau}(j)}(r_{\wt{\tau}(j)})} \Big(\frac{{2\pi}r_{\wt{\tau}(j)}^2}{\Lap Q (r_{\wt{\tau}(j)})}\Big)^{\frac{1}{2}} \Big( 1 + \frac{1}{2N}\mathfrak{B}_1(r_{\wt{\tau}(j)}) + O(N^{-2})\Big).
\end{equation}
\end{itemize}
\end{lem}

\begin{proof}
It suffices to show the first assertion as the second one follows by replacing $N$ with $2N.$ 

Write $\delta_N:= \log N/\sqrt{N}$. 
As seen in \eqref{V tau}, \eqref{r tau min} and \eqref{drdtau}, the function $V_{\tau}$ has a global minimum at $r=r_{\tau}$ and 
$r_{\tau(j+\frac{1}{2})}-r_{\tau(j)}=O(N^{-1})$.
We have
\begin{equation*}
  V_{\tau(j+\frac{1}{2})}(r) \geq V_{\tau(j+\frac{1}{2})}(r_{\tau(j)}+\delta_N) = V_{\tau(j+\frac{1}{2})}(r_{
\tau(j+\frac{1}{2})}) + \frac{1}{2}V_{\tau(j+\frac{1}{2})}''(r_{\tau(j+\frac{1}{2})})(r_{\tau(j)}+\delta_N - r_{\tau(j+\frac{1}{2})})^2 + O(\delta_N^3)
\end{equation*}
for all $r$ with $r>r_{\tau(j)}+\delta_N$ since $rq'(r)$ is increasing in $(0,\infty)$. A similar estimate holds for $r<r_{\tau(j)}-\delta_N$. 
Thus we deduce from the estimate
\begin{equation}
V_{\tau(j+\frac{1}{2})}(r_{\tau(j+\frac{1}{2})}) - V_{\tau(j)}(r_{\tau(j)}) = O(N^{-1})  
\end{equation}
that there exists a positive number $c>0$ such that for all $|r-r_{\tau(j)}|>\delta_N$ 
\begin{align}\label{Vtau:diff}
  V_{\tau(j+\frac{1}{2})}(r) - V_{\tau(j)}(r_{\tau(j)}) \geq c\,\delta_N^2. 
\end{align}

We split $h_j$ in \eqref{hj hj tilde def} into two integrals 
\begin{align*}
  h_j = \int_{0}^\infty e^{-NV_{\tau(j)}(r)} 2r\, dr = \int_{|r-r_{\tau(j)}|<\delta_N} 
  e^{-NV_{\tau(j)}(r)} 2r \, dr
  + \int_{|r-r_{\tau(j)}| > \delta_N} e^{-NV_{\tau(j)}(r)} 2r \, dr.
\end{align*}
Using \eqref{Q growth condition}, we choose sufficiently large $M>0$ such that \begin{equation} \label{M choose}
\int_0^\infty \Big(r^2 e^{-q(r)} \Big)^M \,r \,dr  < +\infty.  
\end{equation} 
We then use \eqref{Vtau:diff} and \eqref{M choose} to find an error estimate for the second integral
\begin{align}
\begin{split} \label{ex_error}
&\quad \int_{|r-r_{\tau(j)}| > \delta_N}  e^{-N V_{\tau(j)}(r) }2r\, dr  
  = e^{-NV_{\tau(j)}(r_{\tau(j)})}
  \int_{|r-r_{\tau(j)}| > \delta_N} e^{-N \big(V_{\tau(j)}(r) - V_{\tau(j)}(r_{\tau(j)})  \big)} \,2r\,  dr
  \\ 
  &\leq e^{-NV_{\tau(j)}(r_{\tau(j)})}e^{-c(N-M)\delta_N^2}
  \int_{|r-r_{\tau(j)}| > \delta_N} e^{-M \big(V_{\tau(j)}(r) - V_{\tau(j)}(r_{\tau(j)})\big)} \,2r \, dr
  = e^{-NV_{\tau(j)}(r_{\tau(j)})} \epsilon_{N}, 
\end{split}
\end{align}
where $\epsilon_{N} = O(e^{-c(\log N)^2})$ for some $c>0$ and the $O$-constants are bounded uniformly for all $\tau$ with $0\leq \tau \leq 1$.
We deduce from the asymptotic expansion of $V_\tau(r)$ near the critical point $r_\tau$ and \eqref{ex_error} that 
\begin{align*}
   h_j =   e^{-NV_{\tau(j)}(r_{\tau(j)})} \Big(\int_{-\delta_N}^{\delta_N}  e^{-2N\Lap Q(r_{\tau(j)})r^2}e^{-N(\frac{1}{3!}V_{\tau(j)}^{(3)}(r_{\tau(j)})r^3 + \frac{1}{4!}V_{\tau(j)}^{(4)}(r_{\tau(j)})r^4 + O(r^5)) } 2(r_{\tau(j)}+r) \,dr + \epsilon_N\Big).
\end{align*}

A change of variables gives that 
\begin{equation}
\begin{split} \label{hj asymptotic 1}
 & e^{NV_{\tau(j)}(r_{\tau(j)})} h_j \\
 =& \frac{2\,r_{\tau(j)}}{\sqrt{N}} 
\int_{-\sqrt{N}\delta_N}^{\sqrt{N}\delta_N} e^{-2\Lap Q(r_{\tau(j)})r^2-\frac{1}{\sqrt{N}}\frac{1}{3!}V_{\tau(j)}^{(3)}(r_{\tau(j)})r^3 - \frac{1}{N}\frac{1}{4!}V_{\tau(j)}^{(4)}(r_{\tau(j)})r^4 + O(N^{-\frac{3}{2}}r^5) } \big(1+\frac{r}{r_{\tau(j)}\sqrt{N}}\big) \,dr + \epsilon_N.
\end{split}
\end{equation}
Using the Taylor series expansion of the function
\begin{align*}
  e^{-\frac{1}{\sqrt{N}}\frac{1}{3!}V_{\tau(j)}^{(3)}(r_{\tau(j)})r^3 - \frac{1}{N}\frac{1}{4!}V_{\tau(j)}^{(4)}(r_{\tau(j)})r^4 },
\end{align*}
we have the asymptotic expansion
\begin{equation}
  \begin{split}
    & \frac{\sqrt{N}}{2r_{\tau(j)}}e^{NV_{\tau(j)}(r_{\tau(j)})}h_j \\
    =& \int_{-\infty}^{\infty} e^{-2\Lap Q(r_{\tau(j)})r^2}\Big[1 - \frac{1}{N}\Big(\frac{1}{6r_{\tau(j)}} V_{\tau(j)}^{(3)}(r_{\tau(j)})+\frac{1}{24}V_{\tau(j)}^{(4)}(r_{\tau(j)}) \Big)r^{4} + \frac{1}{N}\frac{1}{72}(V_{\tau(j)}^{(3)}(r_{\tau(j)}))^2 r^6\Big] \, dr + O(N^{-2})
  \end{split}
\end{equation}
since the odd terms vanish in the integral, leaving only the even terms.

Combining \eqref{V tau diff} with the elementary Gaussian integrals
\begin{equation}
\int_\R e^{-2ar^2}\,dr=\sqrt{\frac{\pi}{2}} \frac{1}{a^{1/2}} , \qquad \int_\R e^{-2ar^2}r^4\,dr=\sqrt{\frac{\pi}{2}} \frac{3}{16} \frac{1}{a^{5/2}} , \qquad \int_\R e^{-2ar^2}r^6\,dr=\sqrt{\frac{\pi}{2}} \frac{15}{64} \frac{1}{a^{7/2}},
\end{equation}
we obtain the desired asymptotic behavior after some straightforward computations. 
\end{proof}

The following elementary integration will be helpful later. 

\begin{lem} \label{Lem_B1 integral}
We have 
\begin{equation}
\begin{split}
\int_{S} \mathfrak{B}_1 \, d\mu_Q =F_Q[\mathbb{A}_{r_0,r_1}] -\frac{1}{4} \log \Big( \frac{ \Delta Q(r_1) }{ \Delta Q(r_0) } \Big) + \frac{1}{3} \log\Big( \frac{r_1}{r_0} \Big) , 
\end{split}
\end{equation}
where $F_Q[\mathbb{A}_{r_0,r_1}]$ is given in \eqref{FQ annulus}. 
\end{lem}
\begin{proof}
By \eqref{B1} and \eqref{eq msr Lap Q}, we have
\begin{equation*}
\begin{split}
\int_{S} \mathfrak{B}_1 \, d\mu_Q &= \frac{1}{6} \log\Big( \frac{r_1}{r_0} \Big) -\frac{19}{48} \log \Big( \frac{ \Delta Q(r_1) }{ \Delta Q(r_0) } \Big) - \frac{1}{16} \int_{r_0}^{r_1} \Big[  \frac{ \d_r^2 \Delta Q(r) }{ \Delta Q(r) } -\frac{5}{3} \Big( \frac{ \d_r \Delta Q(r) }{ \Delta Q(r) }\Big)^2 \Big] \,r \,dr. 
\end{split}
\end{equation*}
Then the lemma follows using integration by parts
\begin{equation}
\begin{split} \label{Delta Q partial integral}
&\quad \int_{r_0}^{r_1} \Big[  \frac{ \d_r^2 \Delta Q(r) }{ \Delta Q(r) } -\frac{5}{3} \Big( \frac{ \d_r \Delta Q(r) }{ \Delta Q(r) }\Big)^2 \Big] \,r \,dr 
\\
&=r_1\, \frac{ (\d_r \Delta Q)(r_1) }{ \Delta Q(r_1) } - r_0\, \frac{ (\d_r \Delta Q)(r_0) }{ \Delta Q(r_0) } + \log \Big( \frac{ \Delta Q(r_0) }{ \Delta Q(r_1) } \Big) - \frac{2}{3} \int_{r_0}^{r_1} \Big( \frac{ \d_r \Delta Q(r) }{ \Delta Q(r) }\Big)^2\,r \,dr. 
\end{split}
\end{equation}
\end{proof}

\subsection{Random normal matrix ensemble} \label{Subsection_annulus normal}

In this subsection, we prove Theorem~\ref{Thm_ZN annulus} (i).
By Lemma~\ref{Lem_hj asymptotic}, we have
\begin{equation}\label{Logh}
  \log h_j = - NV_{\tau(j)}(r_{\tau(j)}) + \frac{1}{2}\Big( \log (2\pi r_{\tau(j)}^2) - \log N - \log \Lap Q(r_{\tau(j)})\Big) + \frac{1}{N}\mathfrak{B}_1(r_{\tau(j)}) + O(N^{-2}).
\end{equation}
In the following lemmas, we analyze the asymptotic behavior of the partial sum of each term in \eqref{Logh}.

\begin{lem}\label{Sum:Vtau}
As $N \to \infty$, we have 
\begin{equation}
  \sum_{j=0}^{N-1} V_{\tau(j)}(r_{\tau(j)}) = N I_Q[\mu_Q]-U_{\mu_Q}(0) + \frac{1}{6N}\log \Big(\frac{r_0}{r_1}\Big) + O(N^{-3}).
\end{equation}
\end{lem}
\begin{proof}
The sequence $\tau$ in \eqref{tau j wt tau j} can be extended to the function on $[0,N]$: $\tau(t) = t/N.$ 
Using the Euler-Maclaurin formula \eqref{EMF} and \eqref{tau j wt tau j}, we have
\begin{align}
\begin{split} \label{sum Vtau EM}
  \sum_{j=0}^{N-1} V_{\tau(j)}(r_{\tau(j)}) &= \int_{0}^{N} V_{\tau(t)}(r_{\tau(t)}) \, dt - \frac{1}{2}\big(V_{\tau(N)}(r_{\tau(N)})-V_{\tau(0)}(r_{\tau(0)})\big) \\
  &\quad + \frac{1}{12}\Big[\d_t V_{\tau(t)}(r_{\tau(t)})\big|_{t=N} - \d_t V_{\tau(t)}(r_{\tau(t)})\big|_{t=0} \Big] + O(N^{-3}).
\end{split}
\end{align}
Here, we also used the second Bernoulli number $B_2=1/6$, which can be easily seen from the definition \eqref{Bernoulli number}.
For the first term on the right-hand side of \eqref{sum Vtau EM}, the change of variables $s=r_{\tau(t)}$, the definition \eqref{V tau} of $V_\tau$, the definition \eqref{tau j wt tau j} of $\tau$, the formula \eqref{drdtau} of $dr/d\tau$, and  the equation \eqref{r tau eq} $r_\tau q'(r_\tau)=2\tau$ for $r_\tau$ give that 
\[
  \frac{1}{N}\int_{0}^{N}V_{\tau(t)}(r_{\tau(t)}) \,dt = 2 \int_{r_0}^{r_1} \big(q(s) - sq'(s)\log s \big) s \Lap Q(s)\, ds.
\]
We use the polar coordinate system to represent the first term on the right-hand side of the above equation as    
\[2\int_{r_0}^{r_1} s q(s) \Lap Q(s)\, ds = \int_{S} Q \cdot \Lap Q \,dA.
\]
By \eqref{Delta Q q' q''}, the method of integration by parts, and the relation $r_1 q'(r_1) = 2, q'(r_0)=0$ (see \eqref{r tau eq}), the second term is simplified to 
\[
- 2\int_{r_0}^{r_1} s^2q'(s)\log s   \Lap Q(s)\, ds = 
 - \frac{1}{4}\int_{r_0}^{r_1} \log s \cdot ((sq'(s))^2)'\, ds 
  \\
  = - \log r_1 + \frac{1}{4} \int_{r_0}^{r_1} s (q'(s))^2 \,ds.
\]
Applying the method of integration by parts again to the last integral,
\[
\frac{1}{4} \int_{r_0}^{r_1} s (q'(s))^2 \,ds = \frac{1}{2} q(r_1) - \frac{1}{4}\int_{r_0}^{r_1} q(s) (sq'(s))' \,ds.
\]
Using the formula \eqref{energy radially sym} of $I_Q[\mu_Q]$, the representation \eqref{Delta Q q' q''} of $\Delta Q$ in terms of $q$, and the equation \eqref{r tau eq} $r_\tau q'(r_\tau)=2\tau$ for $r_\tau$, we have
\begin{align} \label{sum Vtau EM 1}
\begin{split}
  \frac{1}{N}\int_{0}^{N}V_{\tau(t)}(r_{\tau(t)}) \,dt 
  &= \int_{S} Q \cdot \Lap Q \,dA - \log r_1 + \frac{1}{2} q(r_1) - \frac{1}{4}\int_{r_0}^{r_1} q(s) (sq'(s))' \,ds\\
  & = \frac{1}{2}\int_{S} Q \cdot \Lap Q \,dA - \log r_1 + \frac{1}{2} q(r_1)=I_{Q}[\mu_Q] .
\end{split}
\end{align}

For the next term on the right-hand side of \eqref{sum Vtau EM}, we observe that 
\begin{equation} \label{sum Vtau EM 2}
  V_{\tau(N)}(r_{\tau(N)}) - V_{\tau(0)}(r_{\tau(0)}) = V_1(r_1)-V_0(r_0) = q(r_1) - q(r_0)- 2\log r_1=2 U_{\mu_Q}(0),
\end{equation}
where we have used \eqref{potential radially sym}. 
To analyze the remaining term, we use the Leibniz rule and obtain
\begin{align*}
\begin{split}
  \d_t V_{\tau(t)}(r_{\tau(t)})\big|_{t=N} 
  & = \frac{1}{N} \big[\d_{\tau} (q(r_\tau) - 2 \log r_\tau)\big|_{\tau=1} - 2 \log r_1 \big], \\
  \d_t V_{\tau(t)}(r_{\tau(t)})\big|_{t=0} 
  & = \frac{1}{N} \big[\d_{\tau} q(r_\tau) \big|_{\tau=0} - 2 \log r_0 \big].
\end{split}
\end{align*}
It follows from $r_1 q'(r_1) = 2$, $r_0q'(r_0)=0$ and the formula \eqref{drdtau} of $dr/d\tau$ that 
\begin{align}
\begin{split} \label{sum Vtau EM 3}
  &\quad \d_t V_{\tau(t)}(r_{\tau(t)})\big|_{t=N} - \d_t V_{\tau(t)}(r_{\tau(t)})\big|_{t=0} 
  \\
  &= \frac{1}{N} \Big[ \frac{q'(r_1)}{2r_1\Lap Q(r_1)} - 2\log r_1 - \frac{1}{r_1^2 \Lap Q (r_1)} - \frac{q'(r_0)}{2r_0
  \Lap Q(r_0)} + 2\log r_0\Big] = \frac{2}{N} \log\Big( \frac{r_0}{r_1} \Big).
\end{split}
\end{align}
Combining \eqref{sum Vtau EM}, \eqref{sum Vtau EM 1}, \eqref{sum Vtau EM 2}, and \eqref{sum Vtau EM 3}, the proof is complete. 
\end{proof}

\begin{lem}\label{Sum:LapQ}
As $N\to \infty$, we have 
\begin{equation}
\label{Sum:LapQtau}
  \sum_{j=0}^{N-1}\log \Lap Q(r_{\tau(j)}) = N E_{Q}[\mu_Q] - \frac{1}{2} \log \Big( \frac{\Lap Q(r_1)} { \Lap Q(r_0)} \Big) + O(N^{-1}),
\end{equation}
and 
\begin{equation}
\label{Sum:logr}
  \sum_{j=0}^{N-1} \log r_{\tau(j)}  = -N U_{\mu_Q}(0) -\frac{1}{2}\log \Big(\frac{r_1}{r_0}\Big) + O(N^{-1}).
\end{equation}
\end{lem}
\begin{proof}
As in Lemma \ref{Sum:Vtau}, we apply the Euler-Maclaurin formula \eqref{EMF} and obtain 
\begin{align} \label{sum Lap Q EM}
\begin{split}
  \sum_{j=0}^{N-1} \log \Lap Q(r_{\tau(j)}) &= \int_{0}^{N}\log \Lap Q(r_{\tau(t)})\, dt - \frac{1}{2} \log \Big( \frac{\Lap Q(r_1)} { \Lap Q(r_0)} \Big) 
  \\
  &\quad +\frac{1}{12}\Big[ \d_t \log\Lap Q(r_{\tau(t)})\big|_{t=N} - \d_t \log\Lap Q(r_{\tau(t)})\big|_{t=0} \Big] + O(N^{-3}).
\end{split}
\end{align}
It follows from a change of variables and \eqref{drdtau} that the first term on the right-hand side of \eqref{sum Lap Q EM} gives the entropy term 
\begin{equation}
  \int_{0}^{N}\log \Lap Q(r_{\tau(t)}) \,dt = N \int_{r_0}^{r_1} \log \Lap Q(s) 2s\Lap Q(s)\, ds = N \int_{S} \log \Lap Q \,d\mu_Q.
\end{equation}
By the chain rule and the formula \eqref{drdtau} of $dr/d\tau$ again, we also observe that 
\begin{align*}
  \d_t\log\Lap Q(r_{\tau(t)})\big|_{t=N} - \d_t\log\Lap Q(r_{\tau(t)})\big|_{t=0} 
  = \frac{1}{2N}\Big(\frac{(\d_r\Lap Q)(r_1)}{r_1 (\Lap Q (r_1))^2} - \frac{(\d_r\Lap Q)(r_0)}{r_0 (\Lap Q (r_0))^2} \Big)=O(N^{-1}).
\end{align*}
Combining all of the above, we obtain \eqref{Sum:LapQtau}.
The equation \eqref{Sum:logr} follows similarly.
\end{proof}

We are now ready to prove the first assertion of Theorem~\ref{Thm_ZN annulus}.

\begin{proof}[Proof of Theorem~\ref{Thm_ZN annulus} (i)]
Combining Lemmas \ref{Sum:Vtau} and \ref{Sum:LapQ} with \eqref{Logh}, we obtain 
\begin{align*} 
  \sum_{j=0}^{N-1}\log h_j &= -N^2 I_Q[\mu_Q] - \frac{N}{2}\log N + N\Big( \frac{\log (2\pi)}{2} -\frac12 E_{Q}[\mu_Q] \Big) 
  \\
  &\quad -\frac{1}{3}\log\Big(\frac{r_1}{r_0}\Big) + \frac{1}{4}\log \Big( \frac{\Lap Q(r_1)}{\Lap Q(r_0)} \Big) + \int_{S} \mathfrak{B}_1 \, d\mu_Q + O(N^{-1}),
\end{align*}
where $\mathfrak{B}_1$ is given by \eqref{B1}.
Then the desired asymptotic expansion \eqref{ZN expansion cplx annulus} follows from 
\begin{equation} \label{log N!}
\log N!= N \log N-N+\frac12 \log N+\frac12 \log(2\pi)+O(\frac{1}{N}), \qquad (N \to \infty)
\end{equation}
(see e.g. \cite[Eq.(5.11.1)]{olver2010nist}) and Lemma~\ref{Lem_B1 integral}. 
\end{proof}

\subsection{Planar symplectic ensemble} \label{Subsection_annulus symplectic}

In this subsection, we prove Theorem~\ref{Thm_ZN annulus} (ii).

Recall that by Lemma~\ref{Lem_hj asymptotic},
\begin{equation}\label{Logh tilde}
  \log \wt{h}_j = - 2NV_{\wt{\tau}(j)}(r_{\wt{\tau}(j)}) + \frac{1}{2}\Big( \log (\pi r_{\wt{\tau}(j)}^2) - \log N - \log \Lap Q(r_{\wt{\tau}(j)})\Big) + \frac{1}{2N}\mathfrak{B}_1(r_{\wt{\tau}(j)}) + O(N^{-2}).
\end{equation}
In the following lemmas, we derive asymptotic expansions for the partial sum of each term on the right-hand side of \eqref{Logh tilde}. 
We obtain the following as a counterpart of Lemma~\ref{Sum:Vtau}.

\begin{lem}\label{Sum:Vtau symplectic}
As $N\to \infty$, we have
\begin{equation} \label{Vtau sum odd}
  \sum_{j=0}^{N-1} V_{\wt{\tau}(2j+1)}(r_{\wt{\tau}(2j+1)}) =N I_{Q}[\mu_{Q}] -\frac{1}{12N}\log \Big(\frac{r_0}{r_1}\Big)  + O(N^{-3}).
\end{equation}
\end{lem}

\begin{proof}
Applying Lemma~\ref{Sum:Vtau} by replacing $N$ with $2N$, we have 
\begin{equation} \label{Vtau sum 2N}
  \sum_{j=0}^{2N-1} V_{\wt{\tau}(j)}(r_{\wt{\tau}(j)}) = 2N I_Q[\mu_Q] -U_{\mu_Q}(0) + \frac{1}{12N}\log \Big(\frac{r_0}{r_1}\Big) + O(N^{-3}).
\end{equation}
It follows from 
$V_{\wt{\tau}(2j)}(r_{\wt{\tau}(2j)}) = V_{\tau(j)}(r_{\tau(j)})$ and Lemma~\ref{Sum:Vtau} that 
\begin{equation} \label{Vtau sum even}
  \sum_{j=0}^{N-1} V_{\wt{\tau}(2j)}(r_{\wt{\tau}(2j)}) = N I_{Q}[\mu_{Q}] -U_{\mu_Q}(0) + \frac{1}{6N}\log\Big( \frac{r_0}{r_1}\Big) + O(N^{-3}).
\end{equation}
Then \eqref{Vtau sum odd} follows from \eqref{Vtau sum 2N} and \eqref{Vtau sum even}.
\end{proof}

\begin{lem}\label{Sum:LapQ symplectic}
As $N\to \infty$, we have 
\[
 \sum_{j=0}^{N-1}\log \Lap Q(r_{\wt{\tau}(2j+1)}) = N E_{Q}[\mu_Q] + O(N^{-1}),
\]
 and
\[
   \sum_{j=0}^{N-1} \log r_{\wt{\tau}(2j+1)} = -N U_{\mu_Q}(0) + O(N^{-1}).
\]
\end{lem}
\begin{proof}
By Lemma~\ref{Sum:LapQ}, we have 
\begin{align}
\sum_{j=0}^{2N-1}\log \Lap Q(r_{\wt{\tau}(j)}) & = 2N E_{Q}[\mu_Q] - \frac{1}{2} \log\Big( \frac{ \Lap Q(r_1) } { \Lap Q(r_0) } \Big) + O(N^{-1}),\\
  \sum_{j=0}^{2N-1} \log r_{\wt{\tau}(j)} & = -2NU_{\mu_Q}(0) - \frac{1}{2}\log\Big( \frac{r_1}{r_0}\Big) + O(N^{-1}).
\end{align}
Along the lines of Lemma~\ref{Sum:LapQ}, one can also show that 
\begin{align}
 \sum_{j=0}^{N-1}\log \Lap Q(r_{\wt{\tau}(2j)})& = N E_{Q}[\mu_Q] - \frac{1}{2}\log\Big( \frac{ \Lap Q(r_1) } { \Lap Q(r_0) } \Big) + O(N^{-1}),
  \\
   \sum_{j=0}^{N-1} \log r_{\wt{\tau}(2j)}& = -NU_{\mu_Q}(0) -\frac{1}{2}\log \Big(\frac{r_1}{r_0}\Big) + O(N^{-1}).
\end{align}
This completes the proof. 
\end{proof}

We now prove the second assertion of Theorem~\ref{Thm_ZN annulus}.

\begin{proof}[Proof of Theorem~\ref{Thm_ZN annulus} (ii)]

By Lemmas~\ref{Sum:Vtau symplectic} and \ref{Sum:LapQ symplectic}, we have
\begin{align*} 
  \sum_{j=0}^{N-1}\log (2\widetilde{h}_{2j+1}) &=-2N^2 I_Q[\mu_Q] -\frac{N}{2}\log N +\frac{N}{2} \log 2 + \frac{N}{2} \log(2\pi)
  \\
  &\quad -N U_{\mu_Q}(0) -\frac{N}{2} E_Q[\mu_Q] +\frac12 \int_{S} \mathfrak{B}_1 \, d\mu_Q +\frac{1}{6} \log\Big( \frac{r_0}{r_1} \Big)+O(N^{-1}) .
\end{align*}
Combining \eqref{ZN random normal symplectic}, \eqref{log N!}, and Lemma~\ref{Lem_B1 integral} completes the proof. 
\end{proof}

\section{Proof of Theorem~\ref{Thm_ZN disc}} \label{Section_disc}

In this section, we prove Theorem~\ref{Thm_ZN disc}. 
Throughout this section, we let $r_0=0.$

\subsection{Asymptotics of the orthogonal norm}

Let $\delta_N=N^{-1/2}\log N$ and $m_N = N^{\epsilon}$ for $0<\epsilon<1/5$. 
As explained in Subsection~\ref{Subsec_outline}, for the disc droplet case, the asymptotic behaviors of $h_j$ and $\wt{h}_j$ depend on whether the degree $j$ is sufficiently small or not, see Lemmas~\ref{Lem_hj lower} and \ref{Lem_hj higher}, respectively. 

\begin{lem}\label{Lem_hj lower}
As $N \to \infty$, the following holds. 
\begin{itemize}
  \item For $j=0,1,\dots,m_N-1$, we have
\begin{align*}
  \log h_j = - Nq(0) - (j+1) \log \Big(\frac{N}{2}q''(0) \Big) + \log j! + O\big(N^{-\frac{1}{2}} (j+1)^{\frac{3}{2}}(\log N)^{3}  \big).
\end{align*}
 \item 
For $j =0,1,\dots,2m_N-1$, we have
\begin{align*}
  \log \widetilde{h}_j = - 2Nq(0) - (j+1) \log (Nq''(0) ) + \log j! + O\big(N^{-\frac{1}{2}}(j+1)^{\frac{3}{2}}(\log N)^3\big).
\end{align*}
\end{itemize}
\end{lem}
\begin{proof}
The second assertion is an immediate consequence of the first one.
Recall that $\tau(j)=j/N$.
Let
\begin{equation*}
  r^*_\tau := r_\tau \cdot \log N,
\end{equation*}
where $r_\tau$ is given in \eqref{r tau min} or \eqref{r tau eq}, $r_\tau q'(r_\tau) = 2\tau$. 
We consider the decomposition
\begin{align*}
 h_j &= \int_{0}^{\infty} 2r^{2j+1} e^{-N q(r)} \, dr = \int_{0}^{r^*_{{\tau(j+\frac{1}{2})}}} 2r^{2j+1} e^{-Nq(r)}\, dr + \int_{r^*_{\tau(j+\frac{1}{2})}}^\infty 2r^{2j+1} e^{-Nq(r)}\, dr.
\end{align*}
Due to strict subharmonicity of $Q$ in a neighborhood of the droplet, $r_\tau$ defined in \eqref{r tau min} satisfies 
$r_{\tau} = O({\tau}^{\frac{1}{2}})$ as $\tau \to 0$. 
Since the function $V_{\tau}$ in \eqref{V tau} has a global minimum at $r=r_{\tau}$ and increases in $(r_{\tau},\infty)$, for all $r>r^*_{\tau(j+\frac{1}{2})}$, we have
\begin{align} 
\begin{split} \label{V increas h lower}
  V_{\tau(j+\frac{1}{2})}(r) 
  &\geq V_{\tau(j+\frac{1}{2})}(r^*_{\tau(j+\frac{1}{2})}) 
  = q(r^*_{\tau(j+\frac{1}{2})}) - \frac{2j+1}{N}\log(r^*_{\tau(j+\frac{1}{2})})\\
  &= 
  q(0) + \frac{1}{2}q''(0)(r^*_{\tau(j+\frac{1}{2})})^2 - \frac{2j+1}{N}\log(r^*_{\tau(j+\frac{1}{2})}) + O(r^*_{\tau(j+\frac{1}{2})})^3.
\end{split}
\end{align}
Here, we also have used $q'(0)=0$, which follows from \eqref{Delta Q q' q''} and the fact that $\Lap Q(z) \in (0,\infty)$ near the origin. 
Using \eqref{V increas h lower}, it follows that 
\begin{align*}
  &\quad e^{Nq(0)}\int_{r^*_{\tau(j+\frac{1}{2})}}^\infty 2r^{2j+1} e^{-N q(r)}\, dr 
  = \int_{r^*_{\tau(j+\frac{1}{2})}}^\infty 2\,e^{-N( V_{\tau(j+\frac{1}{2})}(r)-q(0))}\, dr
  \\ 
  &\leq e^{-c_1(N-M) (r^*_{\tau(j+\frac{1}{2})} )^2 } (r^*_{\tau(j+\frac{1}{2})})^{(2j+1)\frac{N-M}{N}} \int_{r^*_{\tau(j+\frac{1}{2})}}^\infty 2 \, e^{-M(V_{\tau(j+\frac{1}{2})}(r)-q(0))}\, dr =: \epsilon_N(j)
\end{align*}
for some $c_1>0$. Here, $M$ is given by \eqref{M choose}. 
Since $r_{\tau} = O({\tau}^{\frac{1}{2}})$, there exists $c_2>0$ such that 
\begin{align*}
\epsilon_N(j) = O \Big( (r^*_{\tau(j+\frac{1}{2})})^{(2j+1)\frac{N-M}{N}} \cdot e^{-c_2 (j+\frac{1}{2})(\log N)^2 } \Big) = O\Big( \big(\tfrac{2j+1}{2N}
  (\log N)^2 \big)^{(2j+1)\frac{N-M}{2N}} e^{-c_2(j+\frac{1}{2})(\log N)^2}\Big)  .
\end{align*}
Therefore we obtain 
\begin{align*}
  h_j &= \int_{0}^{\infty} 2r^{2j+1} e^{-N q(r)} \, dr = \int_{0}^{r^*_{\tau(j+\frac{1}{2})}} 2r^{2j+1} e^{-Nq(r)}\, dr + e^{-Nq(0)} \epsilon_N(j)
  \\
    &= \int_{0}^{r^*_{\tau(j+\frac{1}{2})}} 2r^{2j+1} e^{-N(q(0) + \frac{1}{2}q''(0)r^2 ) + O(N (r^*_{\tau(j+\frac{1}{2})})^3 ) } \,dr + e^{-Nq(0)}\epsilon_N(j)
    \\
    &= e^{-Nq(0)} \Big[ N^{-(j+1)} \int_{0}^{\infty} 2r^{2j+1} e^{- \frac{1}{2}q''(0)r^2 } \,dr \, \big(1 + O(N (r^*_{\tau(j+\frac{1}{2})})^3 )\big)+ \epsilon_N(j) \Big]
    \\
    &= e^{-Nq(0)} \Big[ N^{-(j+1)} \Gamma(j+1) \Big(\frac{1}{2}q''(0)\Big)^{-(j+1)}\, \Big(1 + O\Big(N^{-\frac{1}{2}} \big(j+\tfrac{1}{2}\big)^{\frac{3}{2}}(\log N)^3 \Big) \Big)+ \epsilon_N(j) \Big],
\end{align*}
where $\epsilon_N(j)$ is negligible. This completes the proof.
\end{proof}

Recall that the sequences $\tau(j)$ and $\wt{\tau}(j)$ are given by \eqref{tau j wt tau j}: $\tau(j):=j/N, \wt{\tau}(j):=j/(2N)$ and the function $\mathfrak{B}_1$ is given by \eqref{B1}:
\begin{equation*} 
\begin{split}
\mathfrak{B}_1(r):= -\frac{1}{32} \frac{ \d_r^2 \Delta Q(r) }{ (\Delta Q(r))^2 } -\frac{19}{96r} \frac{\d_r \Delta Q(r)}{ (\Delta Q(r))^2 }+\frac{5}{96} \frac{ (\d_r \Delta Q(r))^2 }{ \Delta Q(r)^3 } +\frac{1}{12r^2} \frac{1}{\Delta Q(r)}.
\end{split}
\end{equation*} 
Recall also that $V_\tau$ is given by \eqref{V tau} $V_\tau(r):=q(r)-2\tau\log r$ and $r_{\tau}$ is given by \eqref{r tau min} or \eqref{r tau eq} $r_\tau q'(r_\tau)=2\tau$.
As a counterpart of Lemma~\ref{Lem_hj asymptotic}, we show the following lemma. 

\begin{lem}\label{Lem_hj higher}
As $N \to \infty$, the following holds.
\begin{itemize}
  \item For $j =m_N,m_N+1,\dots, N-1$, we have
\begin{equation*}
  \log h_j = - NV_{\tau(j)}(r_{\tau(j)}) + \frac{1}{2} \log \Big( \frac{2\pi r_{\tau(j)}^2 }{ N \Lap Q(r_{\tau(j)}) } \Big) + \frac{1}{N}\mathfrak{B}_1(r_{\tau(j)}) + O(j^{-\frac{3}{2}}(\log N)^{\alpha}). 
\end{equation*}
\item For $j =2m_N,2m_N+1, \dots, 2N-1$, we have
\begin{equation*}
  \log \widetilde{h}_j = - 2NV_{\wt{\tau}(j)}(r_{\wt{\tau}(j)}) + \frac{1}{2} \log\Big( \frac{\pi r_{\wt{\tau}(j)}^2 }{ N\Lap Q(r_{\wt{\tau}(j)}) } \Big)  + \frac{1}{2N}\mathfrak{B}_1(r_{\wt{\tau}(j)}) + O(j^{-\frac{3}{2}}(\log N)^{\alpha}). 
\end{equation*}
\end{itemize} 
Here $\alpha>0$ is a small constant. 
\end{lem}

\begin{proof}
This lemma can be shown in a similar way to Lemma~\ref{Lem_hj asymptotic}. Recall that $r_{\tau}$ satisfies $r_{\tau}=O(\tau^{\frac{1}{2}})$ as $\tau \to 0$.
Note that for $d\ge1$
\begin{equation}
  |V_{\tau}^{(d)}(r_{\tau})| = |q^{(d)}(r_{\tau}) + (-1)^{d} \, 2\tau\, (d-1)! \, r_{\tau}^{-d}| \leq C_1(1+ \tau^{-\frac{d}{2}+1}),
\end{equation}
where $C_1>0$ is a constant that can be taken uniformly for all $\tau$. We now split the integral for $h_j$ by
\begin{align}
  h_j &= \int_{0}^{\infty} 2r^{2j+1} e^{-N q(r)} \, dr = \int_{0}^{\infty} 2r e^{-N V_{\tau(j)}(r)} \, dr \\
  &= \int_{r_{\tau(j)}-\delta_N}^{r_{\tau(j)}+\delta_N} 2r e^{-NV_{\tau(j)}(r)}\, dr + \int_{|r-r_{\tau(j)}|>\delta_N} 2r e^{-NV_{\tau(j)}(r)}\, dr 
\end{align}
and first compute the integral over the outer region. 
For $m_N \leq j < N$, we have 
\begin{equation*}
  r_{\tau(j+\frac{1}{2})} - r_{\tau(j)} = O((jN)^{-\frac{1}{2}}).
\end{equation*}
Since $V_{\tau}$ is increasing in $(r_{\tau},\infty)$, the Taylor series expansion for $V_{\tau}$ gives 
\begin{align}\label{Vtauh:1}
\begin{split}
    V_{\tau(j+\frac{1}{2})}(r) 
  &\geq V_{\tau(j+\frac{1}{2})}(r_{\tau(j)}+\delta_N) \\
  &= V_{\tau(j+\frac{1}{2})}(r_{\tau(j+\frac{1}{2})}) + \frac{1}{2}V''_{\tau(j+\frac{1}{2})}(r_{\tau(j+\frac{1}{2})})\big(r_{\tau(j)} + \delta_N - r_{\tau(j+\frac{1}{2})} \big)^2 + O(\tau(j)^{-\frac{1}{2}}\delta_N^3)
\end{split}
\end{align}
for all $r>r_{\tau(j)}+\delta_N$. Here, $O(\tau(j)^{-\frac{1}{2}} \delta_N^3) = O(N^{-1}j^{-\frac{1}{2}}(\log N)^3)$. Similarly, since $V_{\tau}$ is decreasing in $(0,r_\tau)$, we have 
\begin{align}\label{Vtauh:1-1}
\begin{split}
    V_{\tau(j+\frac{1}{2})}(r) 
  &\geq V_{\tau(j+\frac{1}{2})}(r_{\tau(j)}-\delta_N) \\
  &= V_{\tau(j+\frac{1}{2})}(r_{\tau(j+\frac{1}{2})}) + \frac{1}{2}V''_{\tau(j+\frac{1}{2})}(r_{\tau(j+\frac{1}{2})})\big(r_{\tau(j)} - \delta_N - r_{\tau(j+\frac{1}{2})} \big)^2 + O(\tau(j)^{-\frac{1}{2}}\delta_N^3)
\end{split}
\end{align}
for all $r< r_{\tau(j)}-\delta_N$. Using the Taylor series for $V_{\tau}$ again, we have 
\begin{align}\label{Vtauh:2}
\begin{split}
  V_{\tau(j+\frac{1}{2})}(r_{\tau(j+\frac{1}{2})}) - V_{\tau(j)}(r_{\tau(j)}) 
  &= V_{\tau(j)}(r_{\tau(j+\frac{1}{2})}) - V_{\tau(j)}(r_{\tau(j)}) - \frac{1}{N}\log r_{\tau(j+\frac{1}{2})} \\ 
  &=O((jN)^{-1}) - \frac{1}{N}\log r_{\tau(j+\frac{1}{2})}.
\end{split}
\end{align}
Thus, it follows from \eqref{Vtauh:1}, \eqref{Vtauh:1-1}, and \eqref{Vtauh:2} that 
\begin{align}\label{int:outer}
\begin{split}
  &\quad e^{NV_{\tau(j)}(r_{\tau(j)})}
  \int_{|r-r_{\tau(j)}|>\delta_N} 2r^{2j+1} e^{-Nq(r)}\,dr 
  = \int_{|r-r_{\tau(j)}|>\delta_N} 2 e^{-N \big(V_{\tau(j+\frac{1}{2})}(r) - V_{\tau(j)}(r_{\tau(j)})\big)} \,dr \\ 
  &\leq e^{-c_1(\log N)^2} e^{\frac{N-M}{N}\log r_{\tau(j+\frac{1}{2})}} \int_{|r-r_{\tau(j)}|>\delta_N} 2 e^{-M\big(V_{\tau(j+\frac{1}{2})}(r) - V_{\tau(j)}(r_{\tau(j)}) \big)}\,dr = O(e^{-c_2(\log N)^2})
\end{split}
\end{align}
 for some $c_1, c_2 >0$. Here $M$ is given by \eqref{M choose}.
 For the integral near the critical point $r_{\tau(j)}$, we use the Taylor series expansion to obtain
\begin{align*}
  &\quad \int_{r_{\tau(j)}-\delta_N}^{r_{\tau(j)}+\delta_N} 2r e^{-NV_{\tau(j)}(r)}\, dr \\
  & = e^{-NV_{\tau(j)}(r_{\tau(j)})}\int_{-\delta_N}^{\delta_N} 2(r_{\tau(j)}+t) \,e^{-N ( \frac{1}{2}V_{\tau(j)}''(r_{\tau(j)})t^2 + \frac{1}{6}V_{\tau(j)}^{(3)}(r_{\tau(j)}) t^3 + \frac{1}{24}V_{\tau(j)}^{(4)}(r_{\tau(j)})t^4 + O(\tau(j)^{-\frac{3}{2}}|t|^5))}  dt \\
  &=e^{-NV_{\tau(j)}(r_{\tau(j)})}\frac{1}{\sqrt{N}}\int_{-\sqrt{N}\delta_N}^{\sqrt{N}\delta_N} 2 e^{-2\Lap Q(r_{\tau(j)})t^2 }  \,\Big(r_{\tau(j)}+\frac{t}{\sqrt{N}}\Big) \\
  & \hspace{10em} \Big( 1 - \frac{V_{\tau(j)}^{(3)}(r_{\tau(j)})}{6\sqrt{N}} t^3 - \frac{V_{\tau(j)}^{(4)}(r_{\tau(j)})}{24 N} t^4 + \frac{(V_{\tau(j)}^{(3)}(r_{\tau(j)}))^2}{72N} t^6 + \epsilon'_{N,1} \Big) \,dt
\end{align*} 
where $\epsilon'_{N,1}=O(j^{-\frac{3}{2}}(\log N)^{\alpha})$ for some $\alpha>0$.
Combining the all of the above, we obtain
\begin{align*}
  h_j = e^{-NV_{\tau(j)}(r_{\tau(j)})}
  \Big[ \frac{1}{\sqrt{N}} \Big(\frac{2\pi r_{\tau(j)}^2}{\Lap Q(r_{\tau(j)})}\Big)^{\frac{1}{2}}\Big(1+\frac{1}{N}\mathfrak{B}_1(r_{\tau(j)})+\epsilon_{N,1}\Big) + \epsilon_{N,2}
  \Big]
\end{align*}
where $\epsilon_{N,1}=O(j^{-\frac{3}{2}}(\log N)^{\alpha})$ and $\epsilon_{N,2} = O(e^{-c_2(\log N)^2})$. 
\end{proof}

\subsection{Random normal matrix ensemble}

In this subsection, we show Theorem~\ref{Thm_ZN disc} (i). 
According to the asymptotic expansions of $h_j$ given in Lemmas~\ref{Lem_hj lower} and ~\ref{Lem_hj higher}, we analyze the summation in \eqref{ZN random normal symplectic} through the decomposition 
\begin{equation} \label{sum hj decomp}
 \sum_{j=0}^{N-1} \log h_j = \sum_{j=0}^{m_N-1} \log h_j + \sum_{j=m_N}^{N-1} \log h_j. 
\end{equation}
The asymptotic behaviors of each summation on the right-hand side of \eqref{sum hj decomp} are given in Lemma~\ref{Lem_sum of log hj lower} and \ref{Lem_sum of log hj higher}, respectively.

\begin{lem} \label{Lem_sum of log hj lower}
As $N\to \infty$, we have
\begin{align*}
  \sum_{j=0}^{m_N-1} \log h_j & = - m_N N q(0) - \frac{m_N(m_N+1)}{2}\Big(\log N + \log \big(\frac{1}{2}q''(0)\big)\Big) 
  \\
  &\quad +\frac{1}{2}m_N^2\log m_N - \frac{3}{4}m_N^2 + \frac{\log(2\pi)}{2}m_N - \frac{1}{12}\log m_N + \zeta'(-1) + O(m_N^{-2}+N^{-\frac{1}{2}(1-5\epsilon)}(\log N)^3).
\end{align*}
\end{lem}

\begin{proof}
By Lemma~\ref{Lem_hj lower}, we have
\begin{align*}
  \sum_{j=0}^{m_N-1} \log h_j = - m_N N q(0) - \frac{m_N(m_N+1)}{2}\Big(\log N + \log \big(\frac{1}{2}q''(0)\big)\Big) + \log G(m_N+1) + O(N^{-\frac{1}{2}(1-5\epsilon)}(\log N)^3),
\end{align*}
where $G$ is the Barnes $G$-function that can be defined recursively by 
\begin{equation} \label{Barnes G def}
G(z+1)=\Gamma(z)G(z),\qquad G(1)=1. 
\end{equation}
Now lemma follows from \eqref{Barnes G asymp}. 
\end{proof}

\begin{lem}\label{Sum:Vtau disc}
As $N\to \infty$, we have
\begin{align*}
  \sum_{j=m_N}^{N-1}V_{\tau(j)}(r_{\tau(j)}) 
  &= N I_{Q}[\mu_Q]-U_{\mu_Q}(0) - \frac{1}{6N}\log r_1 
  \\
  &\quad - m_Nq(0) - \frac{3}{4}\frac{m_N^2}{N} + \frac{1}{2N}\Big(m_N^2-m_N +\frac{1}{6}\Big)\log \Big(\frac{m_N}{N\Lap Q(0)}\Big) + \frac{m_N}{2N} + O(N^{-\frac{1}{2}(3-5\epsilon)}). 
\end{align*}
\end{lem}
\begin{proof}
By applying the Euler-Maclaurin formula \eqref{EMF}, we have
\begin{align}
\begin{split}
  \sum_{j=m_N}^{N-1} V_{\tau(j)}(r_{\tau(j)}) & =  \int_{m_N}^{N} V_{\tau(t)}(r_{\tau(t)}) \,dt - \frac{1}{2}\Big(V_{\tau(N)}(r_{\tau(N)}) - V_{\tau(m_N)}(r_{\tau(m_N)})\Big) 
  \\
  &\quad + \frac{1}{12}\Big(\d_t (V_{\tau(t)}(r_{\tau(t)}))\big|_{t=N} - \d_t ( V_{\tau(t)}(r_{\tau(t)}))\big|_{t=m_N}\Big) + O(N^{-1-2\epsilon}). \label{hj higher 0} 
\end{split}
\end{align}
Here we have used $\d_t^3 (V_{\tau(t)}(r_{\tau(t)}))\big|_{t=m_N} = O(N^{-3} (\tau(m_N))^{-2}) = O(m_N^{-2}N^{-1})$ and $B_2=1/6$. 
By the change of variables $s = r_{\tau(t)}$ and the formula \eqref{energy radially sym} of $I_Q[\mu_Q]$, we obtain 
\begin{align}
 &\quad  \int_{m_N}^N V_{\tau(t)}(r_{\tau(t)})\, dt 
  = 2N\int_{r_{\tau(m_N)}}^{r_1} (q(s) - sq'(s)\log s) s\Lap Q(s) \,ds \\
  & = N\Big( \frac{1}{2}\int_{S \setminus S_{\tau(m_N)}}\!\!\! Q \cdot \Lap Q \, dA - \log r_1 + (\tau(m_N))^2\log r_{\tau(m_N)}+ \frac{1}{2}(q(r_{1}) - \tau(m_N)\cdot q(r_{\tau(m_N)})) \Big)\\
  &= N I_{Q}[\mu_Q] - \frac{N}{2}\int_{S_{\tau(m_N)}} Q \cdot \Lap Q dA + \frac{m_N^2}{N}\log r_{\tau(m_N)} - \frac{m_N}{2}q(r_{\tau(m_N)}).
\end{align}
Observe here that 
\begin{equation}
r_{\tau} = \Big(\frac{2\tau}{q''(0)}\Big)^{\frac{1}{2}} + O(\tau) = \Big(\frac{\tau}{\Lap Q(0)}\Big)^{\frac{1}{2}} + O(\tau)\qquad \textup{as } \tau \to 0. \label{r tau asymp}
\end{equation}
Thus we have
\begin{align}
  \log r_{\tau(m_N)} &= \frac{1}{2} \log \Big( \frac{ \tau(m_N) }{ \Lap Q(0) } \Big) + O(\tau(m_N)^{\frac{1}{2}}) = \frac{1}{2} \log \Big( \frac{ \tau(m_N) }{ \Lap Q(0) } \Big)+ O(N^{-\frac{1}{2}(1-\epsilon)}), \label{log r tau asymp}
  \\ 
  q(r_{\tau(m_N)}) &= q(0) + \frac{1}{2}q''(0)(r_{\tau(m_N)})^2 + O((\tau(m_N))^{\frac{3}{2}}) = q(0) + \tau(m_N) + O(N^{-\frac{3}{2}(1-\epsilon)}), \label{q r tau asymp}
\end{align}
and 
\begin{align*}
  \frac{1}{2}\int_{S_{\tau(m_N)}} Q\cdot \Lap Q \,dA 
  &= \frac{1}{4} \int_{0}^{r_{\tau(m_N)}} q(s) \cdot (sq'(s))' \,ds 
  = \frac{1}{2} \tau(m_N)\cdot q(r_{\tau(m_N)}) - \frac{1}{4}\int_{0}^{r_{\tau(m_N)}} s (q'(s))^2 \,ds \\
  &= \frac{1}{2} \tau(m_N)\cdot q(r_{\tau(m_N)}) - \frac{1}{4}(q''(0))^2 \int_{0}^{r_{\tau(m_N)}} s^3 \,ds + O(r_{\tau(m_N)}^5) \\
  &= \frac{1}{2} \tau(m_N)\cdot q(r_{\tau(m_N)}) - \frac{1}{4} \tau(m_N)^2 + O(N^{-\frac{5}{2}(1-\epsilon)}).
\end{align*}
Combining all of the above asymptotic expansions, we obtain
\begin{align}
\begin{split}
  \int_{m_N}^N V_{\tau(t)}(r_{\tau(t)})\, dt 
  &= NI_{Q}[\mu_Q] - m_N q(r_{\tau(m_N)}) + \frac{1}{4} \frac{m_N^2}{N} + \frac{m_N^2}{N}\log r_{\tau(m_N)} + O(N^{-\frac{1}{2}(3-5\epsilon)}) 
  \\
  &= NI_{Q}[\mu_{Q}] - m_N q(0) - \frac{3}{4}\,\frac{m_N^2}{N} + \frac{m_N^2}{2N}\log\Big( \frac{m_N}{N\Lap Q(0)}\Big) + O(N^{-\frac{1}{2}(3-5\epsilon)}). \label{hj higher 1}
\end{split}
\end{align}
Furthermore, it follows from the formula~\eqref{potential radially sym} of $U_{\mu_Q}(0)$, \eqref{log r tau asymp} and \eqref{q r tau asymp} that 
\begin{align}
\begin{split}
 &\quad  V_{\tau(N)}(r_{\tau(N)}) - V_{\tau(m_N)}(r_{\tau(m_N)}) 
  = q(r_1)- 2\log r_1 - q(r_{\tau(m_N)}) + 2\tau(m_N)\cdot\log r_{\tau(m_N)}
  \\
  &= q(r_1) - 2\log r_1 - q(0) - \frac{m_N}{N} + \frac{m_N}{N}\log \Big( \frac{m_N}{N\Lap Q(0)} \Big)+ O(N^{-\frac{3}{2}(1-\epsilon)}) \label{hj higher 2}
  \\
   &= 2 U_{\mu_Q}(0)- \frac{m_N}{N} + \frac{m_N}{N}\log \Big( \frac{m_N}{N\Lap Q(0)} \Big)+ O(N^{-\frac{3}{2}(1-\epsilon)}).
\end{split}
\end{align} 
Similarly, we have
\begin{align}
\begin{split}
  \d_t (V_{\tau(t)}(r_{\tau(t)}))\big|_{t=N} - \d_t ( V_{\tau(t)}(r_{\tau(t)}))\big|_{t=m_N} 
  &= \frac{2}{N}(\log{r_{\tau(m_N)}}-\log{r_1})\\
  &= \frac{1}{N}\log \Big( \frac{m_N}{N\Lap Q(0)} \Big) - \frac{2}{N}\log r_1 + O(N^{-\frac{1}{2}(3-\epsilon)}). \label{hj higher 3}
\end{split}
\end{align}
Now lemma follows from \eqref{hj higher 0}, \eqref{hj higher 1}, \eqref{hj higher 2} and \eqref{hj higher 3}. 
\end{proof}

\begin{lem}\label{Lem_sum of Lap Q r disc}
As $N \to \infty$, we have
\begin{equation}
 \sum_{j=m_N}^{N-1} \log \Lap Q(r_{\tau(j)})  = NE_{Q}[\mu_Q] - \frac{1}{2} \log\Big( \frac{ \Lap Q(r_1) }{ \Lap Q(0) } \Big) - m_N\log \Lap Q(0) + O(N^{-\frac{1}{2}(1-\epsilon)}),
\end{equation}
and 
\begin{equation}
 \sum_{j=m_N}^{N-1} \log r_{\tau(j)}  = -NU_{\mu_Q}(0) + \frac{m_N}{2} - \frac{1}{2}\log r_1 - \Big(\frac{m_N}{2} -\frac{1}{4} \Big)\log \Big( \frac{m_N}{N\Lap Q(0)} \Big) + O(N^{-\epsilon}).
\end{equation}
\end{lem}
\begin{proof}
Using the Euler-maclaurin formula \eqref{EMF}, 
\begin{align*}
  \sum_{j=m_N}^{N-1} \log \Lap Q(r_{\tau(j)}) 
  =& N\int_{S\setminus S_{\tau(m_N)}} \log \Lap Q\, d\mu_Q - \frac{1}{2} \log\Big( \frac{\Lap Q(r_1)}{ \Lap Q(r_{\tau(m_N)}) } \Big) 
  \\
  &+ \frac{1}{12}\big(\d_t \log \Lap Q (r_{\tau(t)}))\big|_{t=N} - \d_t \log \Lap Q (r_{\tau(t)}))\big|_{t=m_N}\big) + o(N^{-1}\tau(m_N)^{-\frac{1}{2}}).
\end{align*}
We verify from \eqref{r tau asymp} that
\begin{align*}
  \int_{S_{\tau(m_N)}} \log \Lap Q \,d\mu_Q 
  &= \Lap Q(0)\cdot \log\Lap Q(0) \cdot r_{\tau(m_N)}^2 + O(\tau(m_N)^3) = \frac{m_N}{N}\log\Lap Q(0) + O(N^{-3(1-\epsilon)})
\end{align*}
and 
\begin{align*}
  \d_{t} \log \Lap Q(r_{\tau(t)})\big|_{t=m_N} = \frac{1}{2N}\frac{\d_r \Lap Q(r_{\tau(m_N)})}{r_{\tau(m_N)}(\Lap Q (r_{\tau(m_N)}))^2} = O(N^{-1} \tau(m_N)^{-\frac{1}{2}}) = O(N^{-\frac{1}{2}(1+\epsilon)}).
\end{align*}
Observe that $\log \Lap Q(r_{\tau(m_N)}) = \log \Lap Q(0) + O(r_{\tau(m_N)})$. 
Combining above equations, we obtain the first assertion.
Similarly, by using the Euler-Maclaurin formula \eqref{EMF}, \eqref{log r tau asymp}, and \eqref{q r tau asymp}, we have 
\begin{align*}
  \sum_{j=m_N}^{N}\log r_{\tau(j)} 
  &= N\int_{S\setminus S_{\tau(m_N)}}\log|z|\, d\mu_{Q} - \frac{1}{2}\log\Big( \frac{ r_1 }{  r_{\tau(m_N)} } \Big) +O(m_N^{-1})\\
  &=\frac{N}{2}\Big(2\log r_1 - 2\tau(m_N)\log r_{\tau(m_N)} - q(r_1) + q(r_{\tau(m_N)})\Big) - \frac{1}{2}\log\Big( \frac{ r_1 }{  r_{\tau(m_N)} } \Big)
  \\
  &= N\Big(\log r_1 - \frac{q(r_1)-q(0)}{2} +\frac{1}{2}\tau(m_N) \Big) - \frac{1}{2}\log r_1 - \Big(\frac{m_N}{2} -\frac{1}{4}\Big)\log \Big( \frac{\tau(m_N)}{\Lap Q(0)}\Big) + O(N^{-\epsilon}),
\end{align*}
which completes the proof. 
\end{proof}

\begin{lem} \label{Lem_sum of B1 higher} 
As $N\to \infty$, we have
\begin{align*}
  \frac{1}{N}\sum_{j=m_N}^{N-1} \mathfrak{B}_1(r_{\tau(j)}) & = F_Q[\mathbb{D}_{r_1}] + \frac{1}{3}\log r_1 - \frac{1}{12}\log \Big( \frac{m_N}{N}\Big) -\frac{1}{4} \log \Big( \frac{ \Delta Q(r_1) }{ \Delta Q(0) } \Big)
  \\
  &\quad +\frac{1}{6}\log \Lap Q(0) + O(N^{-\epsilon}+N^{-\frac{1}{2}(1-\epsilon)}),
\end{align*}
where $F_Q[\mathbb{D}_{r_1}]$ is given in \eqref{FQ disc}. 
\end{lem}
\begin{proof}
Observe that
\begin{align*}
  \sum_{j=m_N}^{N-1} \mathfrak{B}_1(r_{\tau(j)}) 
  &= N\int_{S\setminus S_{\tau(m_N)}} \mathfrak{B}_1\, d\mu_Q + O(\tau(m_N)^{-1}).
\end{align*}
By \eqref{r tau asymp} and \eqref{log r tau asymp}, 
\begin{align*}
  &\quad \int_{S\setminus S_{\tau(m_N)}} 
  \!\!\!\mathfrak{B}_1 \, d\mu_Q
  = \frac{1}{6}\log \Big(\frac{r_1}{r_{\tau(m_N)}}\Big) - \frac{19}{48}\log \Big(\frac{\Lap Q(r_1)}{\Lap Q(r_{\tau(m_N)})}\Big)-\frac{1}{16} \int_{r_{m_N}}^{r_1} \Big[ \frac{\d_r^2 \Lap Q(r)}{\Lap Q(r)} - \frac{5}{3}\Big(\frac{\d_r\Lap Q(r)}{\Lap Q(r)}\Big)^2 \Big] r\, dr\\
  &= \frac{1}{6}\log r_1 - \frac{1}{12}\log \Big(\frac{\tau(m_N)}{\Lap Q(0)}\Big) - \frac{19}{48}\log \Big(\frac{\Lap Q(r_1)}{\Lap Q(0)}\Big)- 
  \frac{1}{16} \int_{0}^{r_1} \Big[ \frac{\d_r^2 \Lap Q(r)}{\Lap Q(r)} - \frac{5}{3}\Big(\frac{\d_r\Lap Q(r)}{\Lap Q(r)}\Big)^2 \Big] r\, dr + O(N^{-\frac{1}{2}(1-\epsilon)}).
\end{align*}
Thus we have
\begin{align*}
  \frac{1}{N}\sum_{j=m_N}^{N-1} \mathfrak{B}_1(r_{\tau(j)}) & = \frac{1}{6}\log r_1 - \frac{1}{12}\log \Big( \frac{m_N}{N\Lap Q(0)}\Big) - \frac{19}{48}\log \Big(\frac{\Lap Q(r_1)}{\Lap Q(0)}\Big) \\
  &\quad - \frac{1}{16} \int_{0}^{r_1} \Big[ \frac{\d_r^2 \Lap Q(r)}{\Lap Q(r)} - \frac{5}{3}\Big(\frac{\d_r\Lap Q(r)}{\Lap Q(r)}\Big)^2 \Big] r\, dr + O(N^{-\epsilon}+N^{-\frac{1}{2}(1-\epsilon)}).
\end{align*}
Now lemma follows from \eqref{Delta Q partial integral}. 
\end{proof}

\begin{lem} \label{Lem_sum of log hj higher}
As $N\to \infty$, we have
\begin{align*}
  \sum_{j=m_N}^{N-1}\log h_j 
  &= -N^2 I_Q[\mu_Q] + \frac{N-m_N}{2}\log \Big(\frac{2\pi}{N}\Big) + Nm_N q(0) + \frac{3}{4}m_N^2 - \frac{1}{2}NE_{Q}[\mu_Q] 
  \\
  &\quad  - \frac{1}{2} m_N (m_N+1 )\log\Big( \frac{1}{\Lap Q(0)}\Big) + F_Q[\mathbb{D}_{r_1}] 
-\frac{1}{2}\Big(m_N^2 - \frac{1}{6}\Big)\log \Big( \frac{m_N}{N} \Big) + O(N^{-\frac{1}{12}}(\log N)^{3}).
\end{align*}
\end{lem}
\begin{proof}
By Lemma~\ref{Lem_hj higher}, we have
\begin{align*}
  \sum_{j=m_N}^{N-1} \log h_j = & \sum_{j=m_N}^{N-1} \Big(-NV_{\tau(j)}(r_{\tau(j)}) + \log r_{\tau(j)} - \frac{1}{2}\log\Lap Q(r_{\tau(j)}) +\frac{1}{N}\mathfrak{B}_1(r_{\tau(j)}) \Big) 
  \\
  & +\frac{(N-m_N)}{2}\log \Big( \frac{2\pi}{N} \Big)+ O(m_N^{-\frac{1}{2}}(\log N)^{\alpha})+O(N^{-\frac{1}{2}(1-5\epsilon)}).
\end{align*}
Now Lemmas~\ref{Lem_hj higher}, \ref{Sum:Vtau disc}, \ref{Lem_sum of Lap Q r disc} and \ref{Lem_sum of B1 higher} complete the proof.
Here, for the error term, we take $\epsilon = 1/6$ so that $\epsilon/2=(1-5\epsilon)/2=1/12$. 
\end{proof}

We are now ready to prove the first assertion of Theorem~\ref{Thm_ZN annulus}. 

\begin{proof}[Proof of Theorem~\ref{Thm_ZN disc} (i)]
By combining Lemmas~\ref{Lem_sum of log hj lower}, \ref{Lem_sum of log hj higher} and \eqref{sum hj decomp}, we obtain
\begin{equation}
  \sum_{j=0}^{N-1}\log h_j 
   = -N^2 I_Q[\mu_Q] - \frac{1}{2}N\log N - \frac{N}{2}\Big(E_Q[\mu_Q] -\log(2\pi) \Big) - \frac{\log N}{12} +\zeta'(-1) +F_Q[\mathbb{D}_{r_1}] + O(N^{-\frac{1}{12}}(\log N)^{3}).
\end{equation}
Note here that all the terms involving $m_N$ in Lemmas~\ref{Lem_sum of log hj lower} and \ref{Lem_sum of log hj higher} vanish. 
Then the desired asymptotic expansion \eqref{ZN expansion cplx disc} follows from \eqref{ZN random normal symplectic} and \eqref{log N!}. 
This completes the proof.
\end{proof}

\subsection{Planar symplectic ensemble}

In this subsection, we prove the second assertion of Theorem~\ref{Thm_ZN disc}.

As a counterpart of Lemma~\ref{Lem_sum of log hj lower}, we have the following. 
\begin{lem} \label{Lem_sum of log hj lower symplectic}
As $N\to\infty$, we have
\begin{align*}
  \sum_{j=0}^{m_N-1} \log \widetilde{h}_{2j+1} & = -2 m_N N q(0) +m_N^2 \log m_N +\Big(\log 2-\frac32-\log (Nq''(0)) \Big) m_N^2
  \\
  &\quad+\frac12 m_N \log m_N+\Big( \log 2-\frac12+\frac12 \log \pi-\log (Nq''(0)) \Big)m_N
\\
&\quad -\frac{1}{24}\log m_N   + \frac{5}{24}\log 2 +\frac12 \zeta'(-1)+O(m_N^{-1}+N^{-\frac{1}{2}(1-5\epsilon)}(\log N)^3).
\end{align*}
\end{lem}

\begin{proof}
By Lemma~\ref{Lem_hj lower} and \eqref{Gamma multiplication}, we have
\begin{align*}
  \log \widetilde{h}_{2j+1} = - 2Nq(0) - (2j+2) \log (N q''(0)) + \log \Big( \frac{ 2^{2j+1} }{ \sqrt{\pi} } \Gamma(j+1)\Gamma(j+\tfrac32) \Big) + O(N^{-\frac{1}{2}(1-3\epsilon)}).
\end{align*}
Thus we have 
\begin{align*}
  \sum_{j=0}^{m_N-1} \log \widetilde{h}_{2j+1} & = -2 m_N N q(0) - m_N(m_N+1)\log (N q''(0))
   +m_N^2 \log 2-\frac{m_N}{2} \log \pi 
  \\
  &\quad +\log \Big( G(m_N+1) \frac{ G(m_N+\frac32) }{ G(\frac32) } \Big)+ O(N^{-\frac{1}{2}(1-5\epsilon)}).
\end{align*}
Now lemma follows from the asymptotic expansion \eqref{Barnes G asymp} of the Barnes $G$ function and
\begin{equation} \label{G(1/2)}
G(\tfrac12)=2^{ \frac{1}{24} }\,\exp\Big( \frac32 \zeta'(-1) \Big) \pi^{-\frac14}, \qquad G(\tfrac32)=G(\tfrac12)\Gamma(\tfrac12)= G(\tfrac12)\sqrt{\pi}.
\end{equation}
\end{proof}

\begin{lem}\label{Sum:Vtau disc symplectic}
As $N\to \infty,$ we have
\begin{align*}
  \sum_{j=m_N}^{N-1}V_{\wt{\tau}(2j+1)}(r_{\wt{\tau}(2j+1)}) 
  &= N I_{Q}[\mu_Q] + \frac{1}{12N}\log r_1 \\
  &\quad - m_Nq(0)- \frac{3}{4}\frac{m_N^2}{N} +\frac{1}{4N}\Big(2m_N^2-\frac{1}{6}\Big)\log \Big(\frac{m_N}{N\Lap Q(0)}\Big) + O(N^{-\frac{1}{2}(3-5\epsilon)}). 
\end{align*}
\end{lem}
\begin{proof}
By using Lemma~\ref{Sum:Vtau disc} with $N \to 2N$, we have 
\begin{align*}
  \sum_{j=2m_N}^{2N-1}V_{\wt{\tau}(j)}(r_{\wt{\tau}(j)}) 
  &= 2N I_{Q}[\mu_Q] -U_{\mu_Q}(0) - \frac{1}{12N}\log r_1 \\
  &\quad - 2m_Nq(0) - \frac{3}{2}\frac{m_N^2}{N} + \frac{1}{4N}\Big(4m_N^2-2m_N +\frac{1}{6}\Big)\log \Big(\frac{m_N}{N\Lap Q(0)}\Big) + \frac{m_N}{2N} + O(N^{-\frac{1}{2}(3-5\epsilon)}). 
\end{align*}
On the other hand, by the Euler-Maclaurin formula \eqref{EMF}, we have
\begin{align}
\begin{split} 
  \sum_{j=m_N}^{N-1} V_{\wt{\tau}(2j)}(r_{\wt{\tau}(2j)}) &= \frac12\int_{2m_N}^{2N} V_{\wt{\tau}(t)}(r_{\wt{\tau}(t)}) \, dt - \frac{1}{2}\big(V_{\wt{\tau}(2N)}(r_{\wt{\tau}(2N)})-V_{\wt{\tau}(2m_N)}(r_{\wt{\tau}(2m_N)})\big)
  \\
  &\quad + \frac{1}{12}\Big(\d_t (V_{\wt{\tau}(2t)}(r_{\wt{\tau}(2t)}))\big|_{t=N} - \d_t ( V_{\wt{\tau}(2t)}(r_{\wt{\tau}(2t)}))\big|_{t=m_N}\Big) + O(N^{-1-2\epsilon}).
\end{split}
\end{align}
Following the proof of Lemma~\ref{Sum:Vtau disc}, we have
\begin{align}
\begin{split} 
  \sum_{j=m_N}^{N-1} V_{\wt{\tau}(2j)}(r_{\wt{\tau}(2j)}) &= N I_{Q}[\mu_Q] -U_{\mu_Q}(0) - \frac{1}{6N}\log r_1 \\
  &\quad - m_Nq(0) - \frac{3}{4}\frac{m_N^2}{N} + \frac{1}{2N}(m_N^2-m_N +\frac{1}{6})\log \Big(\frac{m_N}{N\Lap Q(0)}\Big) + \frac{m_N}{2N} + O(N^{-\frac{1}{2}(3-5\epsilon)}), 
\end{split}
\end{align}
which completes the proof.
\end{proof}

\begin{lem}\label{Lem_sum of Lap Q r disc symplectic}
As $N \to \infty$, we have
\begin{equation}
\sum_{j=m_N}^{N-1} \log \Lap Q(r_{\wt{\tau}(2j+1)}) = NE_{Q}[\mu_Q]- m_N\log \Lap Q(0) + O(N^{-\frac{1}{2}(1-\epsilon)}),
\end{equation}
and 
\begin{equation}
 \sum_{j=m_N}^{N-1} \log r_{\wt{\tau}(2j+1)}  =-NU_{\mu_Q}(0) + \frac{m_N}{2} - \frac{m_N}{2} \log \Big(\frac{m_N}{N\Lap Q(0)}\Big) + O(N^{-\epsilon}).
\end{equation}
\end{lem}
\begin{proof}
By Lemma~\ref{Lem_sum of Lap Q r disc} with $N \to 2N$, we have 
\begin{align*}
  \sum_{j=2m_N}^{2N-1} \log \Lap Q(r_{\wt{\tau}(j)}) &= 2NE_{Q}[\mu_Q] - \frac{1}{2} \log \Big( \frac{\Lap Q(r_1)}{ \Lap Q(0) } \Big) - 2m_N\log \Lap Q(0) + O(N^{-\frac{1}{2}(1-\epsilon)}),
  \\
  \sum_{j=2m_N}^{2N-1} \log r_{\wt{\tau}(j)} &= -2N U_{\mu_Q}(0) + m_N - \frac{1}{2}\log r_1 - \Big(m_N -\frac{1}{4} \Big)\log \Big(\frac{m_N}{N\Lap Q(0)}\Big) + O(N^{-\epsilon}).
\end{align*}
Following the proof of Lemma~\ref{Lem_sum of Lap Q r disc}, we also have 
\begin{align*}
  \sum_{j=m_N}^{N-1} \log \Lap Q(r_{\wt{\tau}(2j)}) & = NE_{Q}[\mu_Q] -\frac{1}{2} \log \Big( \frac{\Lap Q(r_1)}{ \Lap Q(0) } \Big) - 2m_N\log \Lap Q(0) + O(N^{-\frac{1}{2}(1-\epsilon)}),
  \\
  \sum_{j=m_N}^{N-1} \log r_{\wt{\tau}(2j)} &= -NU_{\mu_Q}(0)+ m_N - \frac{1}{2}\log r_1 - \Big( m_N -\frac{1}{4} \Big)\log \Big( \frac{m_N}{N\Lap Q(0)} \Big) + O(N^{-\epsilon}).
\end{align*}
This completes the proof. 
\end{proof}

\begin{lem} \label{Lem_sum of B1 higher symplectic}
As $N\to\infty$, we have
\begin{align*}
  \frac{1}{2N}\sum_{j=m_N}^{N-1} \mathfrak{B}_1(r_{\wt{\tau}(2j+1)}) = \frac12 F_Q[\mathbb{D}_{r_1}] + \frac{1}{6}\log r_1 - \frac{1}{24}\log \Big( \frac{m_N}{N}\Big) -\frac{1}{8} \log \Big( \frac{ \Delta Q(r_1) }{ \Delta Q(0) } \Big)+\frac{1}{12}\log \Lap Q(0) + O(m_N^{-1}). 
\end{align*}
\end{lem}
\begin{proof}
This lemma follows along the same lines of Lemma~\ref{Lem_sum of B1 higher}.
\end{proof}

\begin{lem} \label{Lem_sum of log hj higher symplectic}
As $N\to\infty$, we have
\begin{align*}
  \sum_{j=m_N}^{N-1} \log \widetilde{h}_{2j+1} & =-2N^2 I_{Q}[\mu_Q]-\frac{N}{2}\log N+ N\Big( \frac{\log \pi}{2} -U_{\mu_Q}(0) -\frac12 E_{Q}[\mu_Q] \Big)
  \\
  &\quad + \frac{3}{2} m_N^2 -\Big(m_N^2-\frac{1}{24}\Big)\log \Big(\frac{m_N}{N}\Big) 
  +m_N \Big( 2 Nq(0) + \frac{1}{2} +\frac12 \log \Big(\frac{m_N}{\pi}\Big) \Big)
  \\
  &\quad -\Big(m_N^2+\frac{1}{8}\Big)\log \Big(\frac{1}{\Lap Q(0)}\Big) + \frac12 F_Q[\mathbb{D}_{r_1}] + \frac{1}{8} \log\Big( \frac{ 1 }{ \Delta Q(r_1) } \Big) 
+ O(m_N^{-\frac{1}{2}}(\log N)^{\alpha}).
\end{align*}
\end{lem}
\begin{proof}
Note that by Lemma~\ref{Lem_hj higher}, we have
\begin{align*}
  \sum_{j=m_N}^{N-1} \log \widetilde{h}_{2j+1} = & \sum_{j=m_N}^{N-1} \Big(-2NV_{\wt{\tau}(2j+1)}(r_{\wt{\tau}(2j+1)}) + \log r_{\wt{\tau}(2j+1)} - \frac{1}{2}\log\Lap Q(r_{\wt{\tau}(2j+1)}) +\frac{1}{2N}\mathfrak{B}_1(r_{\wt{\tau}(2j+1)}) \Big) 
  \\
  & +\frac{(N-m_N)}{2}\log \Big(\frac{\pi}{N}\Big) + O(m_N^{-\frac{1}{2}}(\log N)^{\alpha}).
\end{align*}
The lemma now follows from Lemmas~\ref{Sum:Vtau disc symplectic}, \ref{Lem_sum of Lap Q r disc symplectic} and \ref{Lem_sum of B1 higher symplectic}.
\end{proof}

We now finish the proof of Theorem~\ref{Thm_ZN disc}. 

\begin{proof}[Proof of Theorem~\ref{Thm_ZN disc} (ii)]
Combining Lemmas~\ref{Lem_sum of log hj lower symplectic} and \ref{Lem_sum of log hj higher symplectic}, after long but straightforward simplifications, we obtain 
\begin{align*}
&\quad \sum_{j=0}^{N-1} \log (2\widetilde{h}_{2j+1}) =N \log 2+ \sum_{j=0}^{m_N-1} \log \widetilde{h}_{2j+1} + \sum_{j=m_N}^{N-1} \log \widetilde{h}_{2j+1} 
\\
&=-2N^2 I_{Q}[\mu_Q] -\frac{1}{2} N \log N + N\Big( \frac{\log (4\pi)}{2} -U_{\mu_Q}(0) -\frac12 E_{Q}[\mu_Q] \Big) -\frac1{24}\log N
  \\
  &\quad+ \frac{5}{24}\log 2 +\frac12 \zeta'(-1) +\frac12 F_Q[\mathbb{D}_{r_1}] +\frac{1}{8} \log\Big( \frac{ \Lap Q(0) }{ \Delta Q(r_1) } \Big) +O(m_N^{-1}+N^{-\frac{1}{2}(1-5\epsilon)}(\log N)^3).
\end{align*}
Again, it is noteworthy that all the terms in Lemmas~\ref{Lem_sum of log hj lower symplectic} and \ref{Lem_sum of log hj higher symplectic} involving $m_N$ cancel each other. 
Now the asymptotic behavior \eqref{ZN expansion symp disc} follows from \eqref{ZN random normal symplectic} and the asymptotic expansion \eqref{log N!} of $\log N!$ with $\epsilon=1/6$. 
\end{proof}

\section{Examples: Mittag-Leffler ensemble and truncated unitary ensemble} \label{Appendix_Examples}

This section presents examples of our Theorems~\ref{Thm_ZN annulus} and \ref{Thm_ZN disc} for some well-known planar point processes. 

\subsection{Mittag-Leffler ensemble} \label{Subsec_ML}
Let us consider the potential
\begin{equation} \label{Q ML}
	Q(z)=|z|^{2\lambda}-2c \log|z|, \quad c > 0. 
	\end{equation} 
The models \eqref{Gibbs cplx beta} and \eqref{Gibbs symplectic beta} associated with the potential \eqref{Q ML} are known as the Mittag-Leffler ensemble \cite{ameur2018random}. 
We refer to \cite{byun2022characteristic,charlier2021large,charlier2022exponential} and \cite{akemann2021scaling} for recent studies on complex and symplectic Mittag-Leffler ensembles, respectively.
Using \eqref{droplet}, we have
	\begin{equation} \label{density power Q c}
	r_0= \Big(\frac{c}{\lambda}\Big)^{\frac{1}{2\lambda}},\qquad r_1= \Big(\frac{1+c}{\lambda}\Big)^{\frac{1}{2\lambda}}, \qquad \Delta Q(z)= \lambda^2 |z|^{2\lambda-2}.
	\end{equation}
In particular, by \eqref{density power Q c}, the Mittag-Leffler ensemble \eqref{Q ML} falls into the class considered in Theorem~\ref{Thm_ZN annulus}. 

By direct computations using \eqref{entropy}, \eqref{energy radially sym} and \eqref{potential radially sym}, we have
\begin{align}
I_Q[\mu_Q] & = \frac{1}{2\lambda }\log \Big( \frac{c^{c^2} }{(1+c)^{(1+c)^2}} \Big)+\frac{1+2c}{2\lambda}\Big(\log \lambda+\frac32\Big), \label{IQ ML}
\\
E_Q[\mu_Q] & =\frac{1+\lambda}{\lambda} \log \lambda + \frac{1-\lambda}{\lambda} \Big(1+ \log\Big(\frac{c^c}{(1+c)^{1+c} }\Big) \Big). \label{EQ ML}
\end{align}
It also follows from \eqref{FQ annulus} and
\begin{equation}
r\frac{ (\partial_r \Lap) Q (r) }{ \Lap Q (r)}= 2 \lambda-2, \qquad r^2 \Lap Q(r)= \lambda^2 r^{2\lambda}
\end{equation}
that 
\begin{equation}
F_Q[\mathbb{A}_{r_0,r_1}] = \Big( \frac{\lambda}{6}-\frac{(\lambda-1)^2}{6} \Big) \log \Big( \frac{r_0}{r_1}\Big)  =-\frac{ \lambda^2-3\lambda+1 }{12 \lambda}\log\Big( \frac{c}{1+c}\Big). \label{FQ ML} 
\end{equation}

On the other hand, by using \eqref{ZN random normal symplectic},
\begin{equation}
h_j= 2\int_0^\infty r^{2j+1+2cN} e^{-N r^{ 2\lambda }} \,dr= \frac{ 1 }{ \lambda } N^{- \frac{j+1+cN}{ \lambda } } \Gamma\Big( \frac{j+Nc+1}{ \lambda } \Big ) , 
\end{equation}
and the analogous formula for $\wt{h}_j$, we have 
\begin{equation}
Z_N= \frac{ N!} {\lambda^{N} } N^{-\frac{ (2c+1)N^2+N }{2\lambda}  } \prod_{j=0}^{N-1} \Gamma\Big( \frac{j+Nc+1}{\lambda} \Big), \qquad \wt{Z}_N =\frac{ N! }{ (\lambda/2)^N }\, (2N)^{-\frac{ (2c+1)N^2+N }{\lambda}  }  \prod_{j=0}^{N-1} \Gamma\Big( \frac{j+Nc+1}{\lambda/2} \Big).
\end{equation}
Furthermore, using the multiplication theorem of gamma function (\cite[Eq. (5.5.6)]{olver2010nist})
\begin{equation}\label{Gamma multiplication}
\Gamma(nz)= (2\pi)^{ \frac{1-n}{2} } n^{ nz-\frac12 } \prod_{k=0}^{n-1} \Gamma(z+\tfrac{k}{n})
\end{equation}
and the characteristic property \eqref{Barnes G def} of the Barnes $G$-function, we have that for $ \frac{1}{\lambda} \in \mathbb{N}$,
\begin{equation} \label{ZN ML cplx}
Z_N=  N!\, N^{-\frac{ (2c+1)N^2+N }{2\lambda}  } (2\pi)^{ ( \frac12-\frac{1}{2\lambda} ) N } \Big( \frac{1}{\lambda} \Big)^{ (\frac{1}{2\lambda}+\frac{c}{\lambda} )N^2+ ( \frac{1}{\lambda}+\frac{1}{2}-\frac{1}{2\lambda} ) N  }  \prod_{k=0}^{ \frac{1}{\lambda}-1} \frac{ G( N+Nc+1+\lambda k)}{ G( Nc+1+\lambda k) }.
\end{equation}
Similarly, for $ \frac{2}{\lambda} \in \mathbb{N}$, we have 
\begin{equation} \label{ZN ML symplectic}
\wt{Z}_N = N! \, (2N)^{-\frac{ (2c+1)N^2+N }{\lambda}  } (2\pi)^{ ( \frac12-\frac{1}{\lambda} ) N } \Big( \frac{2}{\lambda} \Big)^{ (\frac{1}{\lambda}+\frac{2c}{\lambda} )N^2+ ( \frac{2}{\lambda}+\frac{1}{2}-\frac{1}{\lambda} ) N  }  \prod_{k=0}^{ \frac{2}{\lambda}-1} \frac{ G( N+Nc+1+\frac{\lambda}{2} k)}{ G( Nc+1+\frac{\lambda}{2} k) } .
\end{equation}
Then by using \eqref{Barnes G asymp}, one can directly check that the partition functions \eqref{ZN ML cplx} and \eqref{ZN ML symplectic} satisfy the expansions \eqref{ZN expansion cplx annulus} and \eqref{ZN expansion symp annulus} with \eqref{IQ ML}, \eqref{EQ ML} and \eqref{FQ ML}.

\subsection{Truncated unitary ensemble}\label{Subsec_truncated}
We now consider the potential
\begin{equation} \label{Q truncated unitary strong}
Q(z)=
\begin{cases}
-\alpha \, \log \Big( 1-\dfrac{|z|^2}{R^2(1+\alpha)  } \Big) &\textup{if } |z| \le R\sqrt{1+\alpha },
\smallskip 
\\
\infty &\textup{otherwise}, 
\end{cases} \qquad \alpha,R>0. 
\end{equation}
The models associated with \eqref{Q truncated unitary strong} correspond to the truncated unitary ensembles at strong non-unitarity \cite{khoruzhenko2021truncations,MR1748745}.
These models provide a one-parameter generalization of the Ginibre ensembles that can be recovered in the extremal case, i.e. $\lim_{\alpha \to \infty} Q(z)= |z|^2/R^2$. 
(See \cite{dubach2021eigenvector,serebryakov2021characteristic} and references therein for recent works on these models.)
In this case, we have
\begin{equation} \label{density truncated}
r_0=0,\qquad r_1=R, \qquad \Delta Q(z) = \frac{ R^2\, \alpha (1+\alpha ) }{ ( R^2(1+\alpha) -|z|^2)^2}.
\end{equation}
From \eqref{density truncated}, we see that the truncated unitary ensembles are contained in the class covered in Theorem~\ref{Thm_ZN disc}. 
(The hard edge condition imposed in \eqref{Q truncated unitary strong} outside the droplet does not harm the proof of Theorem~\ref{Thm_ZN disc}.)

It is easy to verify from \eqref{entropy}, \eqref{energy radially sym} and \eqref{potential radially sym} that 
\begin{align}
I_Q[\mu_Q] & = -\frac{\alpha }{2}-\frac{\alpha(2+\alpha )}{2} \log \Big( \frac{\alpha }{1+\alpha }\Big) -\log R, \label{IQ trunc}
\\
E_Q[\mu_Q]&=- 2-(1+2\alpha )\log \Big( \frac{\alpha }{1+\alpha } \Big) -2\log R \label{EQ trunc}.
\end{align}
Since 
\begin{equation}
r\frac{ (\partial_r \Lap) Q (r) }{ \Lap Q (r)}\Big|_{r=R} = \frac{4r^2}{R^2(1+\alpha)-r^2} \Big|_{r=R}= \frac{4}{\alpha} , \qquad R^2 \Lap Q(R)= \frac{ 1+\alpha }{ \alpha } , 
\end{equation}
we deduce from \eqref{FQ disc} that 
\begin{equation}
F_Q[\mathbb{D}_{R}]= \frac{1}{12}\log \Big( \frac{\alpha}{1+\alpha}\Big) -\frac{1}{4\alpha} +\frac{1}{3}\Big( \frac{1}{\alpha}+\log \Big( \frac{\alpha}{1+\alpha} \Big) \Big)  = \frac{1}{12}\Big( \frac{1}{\alpha }+5 \log\Big(\frac{\alpha }{1+\alpha }\Big) \Big). \label{FQ trunc}
\end{equation}
Notice here that $F_Q[\mathbb{D}_{R}]$ is independent of $R$, which is consistent with the invariance of the $O(1)$-terms of \eqref{ZN expansion cplx disc} and \eqref{ZN expansion symp disc} under the dilation, see the remark below Theorem~\ref{Thm_ZN disc}. 

Using the Euler's beta integral
\begin{equation}
 \frac{\Gamma(a) \Gamma(b)}{ \Gamma(a+b) } = \int_0^1 t^{a-1} (1-t)^{b-1}\,dt, \qquad (a,b>0), 
\end{equation}
the orthogonal norm $h_j$ is computed as 
\begin{equation}
h_j=\int_0^{ R\sqrt{1+\alpha} } r^{2j+1} \Big( 1- \frac{r^2}{ R^2(1+\alpha) } \Big)^{\alpha N} \,dr= R^{2j+2}(1+\alpha )^{j+1} \frac{ \Gamma(\alpha N+1) \Gamma(j+1) }{ \Gamma(\alpha N+j+2) }.
\end{equation}
Then by \eqref{ZN random normal symplectic} and \eqref{Barnes G def}, we have 
\begin{equation}\label{ZN trunc cplx}
Z_N= N! \,R^{N(N+1)} (1+\alpha )^{ \frac{N^2}{2}+\frac{N}{2} } \, \Gamma(\alpha N+1)^N \, \frac{ G(N+1)G(\alpha N+2) }{ G(\alpha N+N+2) }.
\end{equation}
Similarly, by \eqref{ZN random normal symplectic} and the duplication formula of the gamma function (i.e. \eqref{Gamma multiplication} with $n=2$),
\begin{equation} \label{ZN trunc symplectic}
\begin{split}
\wt{Z}_N & = N! \, R^{2N(N+1)} 2^{-2\alpha N^2}\, (1+\alpha )^{ N^2+N } \, \Gamma(2\alpha N+1)^N 
\\
&\quad \times  G(N+1) \frac{G(N+\frac32)}{G(\frac32)} \frac{G(\alpha N+2)}{G(\alpha N+N+2)}  \frac{G(\alpha N+\frac32)}{ G(\alpha N+N+\frac32)} . 
\end{split}
\end{equation}
Then by using \eqref{Barnes G asymp}, it is again straightforward to check that the partition functions \eqref{ZN trunc cplx} and \eqref{ZN trunc symplectic} satisfy the expansions \eqref{ZN expansion cplx disc} and \eqref{ZN expansion symp disc} with \eqref{IQ trunc}, \eqref{EQ trunc} and \eqref{FQ trunc}. 
In the extremal case where $\alpha \to \infty$, the expansion of the partition function $Z_N$ of the complex Ginibre ensemble appears in \cite{tellez1999exact}.

\subsection*{Acknowledgements}
We thank Christophe Charlier and Thomas Lebl\'{e} for helpful discussions.

\bibliographystyle{abbrv}
\bibliography{RMTbib}

\end{document}